\documentclass[12pt,american]{article}
\usepackage{helvet}
\usepackage{courier}
\usepackage[T1]{fontenc}
\usepackage[latin9]{inputenc}
\usepackage{geometry}
\usepackage{color}
\usepackage{amsmath}
\usepackage{amsthm}
\usepackage{amssymb}
\usepackage{graphicx}

\makeatletter

\providecommand{\tabularnewline}{\\}

\theoremstyle{plain}
\newtheorem{thm}{\protect\theoremname}[section]
  \theoremstyle{definition}
  \newtheorem{defn}[thm]{\protect\definitionname}
  \theoremstyle{plain}
  \newtheorem*{question*}{\protect\questionname}
  \theoremstyle{definition}
  \newtheorem{example}[thm]{\protect\examplename}
  \theoremstyle{plain}
  \newtheorem{cor}[thm]{\protect\corollaryname}
  \theoremstyle{plain}
  \newtheorem{assumption}[thm]{\protect\assumptionname}
  \theoremstyle{remark}
  \newtheorem*{rem*}{\protect\remarkname}

\date{}

\makeatother

\usepackage{babel}
  \providecommand{\assumptionname}{Assumption}
  \providecommand{\corollaryname}{Corollary}
  \providecommand{\definitionname}{Definition}
  \providecommand{\examplename}{Example}
  \providecommand{\questionname}{Question}
  \providecommand{\remarkname}{Remark}
\providecommand{\theoremname}{Theorem}

\begin{document}

\title{Volumetric variational principles for a class of partial differential
equations defined on surfaces and curves  {\small{}}\\
 }

\author{Jay Chu\thanks{National Tsinghua University, Taiwan}~~~and Richard
Tsai\thanks{KTH Royal Institute of Technology, 11044 Stockholm, Sweden and The
University of Texas at Austin, TX 78712, USA}}
\maketitle
\begin{abstract}
In this paper, we propose simple numerical algorithms for partial
differential equations (PDEs) defined on closed, smooth surfaces (or
curves). In particular, we consider PDEs that originate from variational
principles defined on the surfaces; these include Laplace-Beltrami
equations and surface wave equations. The approach is to systematically
formulate  {{} extensions of the variational integrals,
derive the Euler-Lagrange equations of the extended problem, including
the boundary conditions} that can be easily discretized on uniform
Cartesian grids or adaptive meshes. In our approach, the surfaces
are defined implicitly by the distance functions or by the closest
point mapping. As such extensions are not unique, we investigate how
a class of simple extensions can influence the resulting PDEs. In
particular, we reduce the surface PDEs to model problems defined on
a periodic strip and the corresponding boundary conditions, and use
classical Fourier and Laplace transform methods to study the well-posedness
of the resulting problems. For elliptic and parabolic problems, our
boundary closure mostly yields stable algorithms to solve nonlinear
surface PDEs. For hyperbolic problems, the proposed boundary closure
is unstable in general, but the instability can be easily controlled
by either adding a higher order regularization term or by periodically
but infrequently ''reinitializing'' the computed solutions. Some
numerical examples for each representative surface PDEs are presented.
\end{abstract}

\section{Introduction}

This paper targets applications that use implicit or non-parametric
representations of closed surfaces or curves and require numerical
solution of certain partial differential equations defined on these
surfaces. In immiscible multiphase fluids, surfactants can change
important physical properties of the fluid mixture by lowering the
tension of the fluid interface. The surfactant concentration satisfies
a convection-diffusion equation in which the diffusion is described
by a surface Laplacian of the concentration. This is a typical application
in which the level set method enjoys an advantage in tracking the
dynamically changing fluid interface. One of the challenges in developing
a level set method for this application is in discretizing the surface
Laplacian (Laplace-Beltrami) term, which typically requires some extension
of the surfactant concentration and of the surface Laplacian operator
into the ambient space somehow \cite{xu2012level,xu2006level,xu2003eulerian}.
Eigenfunctions and eigenvalues of Laplace-Beltrami operator are of
great interest and use in many scientific disciplines. In computer
vision, eigenfunctions of the Laplace-Beltrami operator are used to
compare and classify surfaces or solid objects \cite{reuter2006laplace,rustamov2007laplace}.
In differential geometry, solutions of the Laplace-Beltrami problem
(Poisson's equation on surfaces) can be used to find the Hodge decompositions
of vector fields defined on the surfaces. The Hodge decomposition
can be used to formulate boundary integral methods for problems in
computational electromagnetic \cite{imbert2017pseudo,o2017second}.\textbf{ }

We aim at developing a general mathematical framework for designing
numerical schemes for solving a wide class of partial differential
equations (PDEs) defined on closed manifolds that are not defined
parametrically. The numerical schemes will inherit the flexibility
of the level set method in dealing with such type of manifolds. 

Let $\Omega\subset\mathbb{R}^{d}$ be a bounded open set with $C^{2}$
boundary $\Gamma=\partial\Omega$. For any small $\epsilon>0,$ define
the narrowband of $\Gamma$ as 
\begin{equation}
T_{\epsilon}=T_{\epsilon}(\Gamma):=\{z\in\mathbb{R}^{d}:\min_{x\in\Gamma}|z-x|<\epsilon\}.\label{eq:T_eps}
\end{equation}
In this paper, we consider integro-differential operators of the form
\begin{equation}
I_{\Gamma}[u]:=\int_{\Gamma}L(x,u(x),\nabla_{\Gamma}u(x))dS(x),\label{eq:elliptic_I}
\end{equation}
or 
\begin{equation}
I_{\Gamma}[u]:=\int_{t_{1}}^{t_{2}}\int_{\Gamma}L(x,u(x,t),u_{t}(x,t),\nabla_{\Gamma}u(x,t))dS(x)dt,\label{eq:hyperbolic_I}
\end{equation}
and their extensions 
\begin{equation}
\tilde{I}_{T_{\epsilon}}[v]:=\int_{T_{\epsilon}}\bar{L}(z,v(z),\nabla v(z))dz,\label{eq:elliptic-extension}
\end{equation}
 or 
\begin{equation}
\tilde{I}_{T_{\epsilon}}[v]:=\int_{t_{1}}^{t_{2}}\int_{T_{\epsilon}}\bar{L}(z,v(z,t),v_{t}(z,t),\nabla v(z,t))dzdt.\label{eq:hyperbolic-extension}
\end{equation}

\noindent We shall focus on partial differential equations arising
from \emph{calculus of variations} and develop an approach for finding
critical points of $I_{\Gamma}[u]$, i.e. 
\[
\delta I_{\Gamma}[u]=0,
\]
by solving 
\[
\delta\tilde{I}_{T_{\epsilon}}[v]=0.
\]
Here $\delta I_{\Gamma}$ and $\delta\tilde{I}_{T_{\epsilon}}$ denote
the variational derivatives of $I_{\Gamma}$ and $\tilde{I}_{T_{\epsilon}}$
in suitable function spaces, we shall further assume that $L$ and
$\bar{L}$ are smooth functions. We first give some definitions that
will facilitate the discussion. 
\begin{defn}
The closest point mapping $P_{\Gamma}:T_{\epsilon}\mapsto\Gamma$
is defined by 
\begin{equation}
P_{\Gamma}z:=\mathrm{argmin}_{x\in\Gamma}|z-x|.\label{eq:closest-point-mapping}
\end{equation}
If more than one point of $\Gamma$ is equidistant to $x$, we shall
randomly assign one of them as the definition of $P_{\Gamma}(z)$.
Since $\Gamma$ is assumed to be $C^{2},$ $P_{\Gamma}$ is well-defined
if $z$ is no farther than $\epsilon$ distance from $\Gamma$, for
any $\epsilon\in[0,\kappa_{\infty}^{-1})$ where $\kappa_{\infty}$
is an upper bound of the curvatures of $\Gamma.$ 
\end{defn}
\noindent If the distance function 
\[
d_{\Gamma}(z):=|P_{\Gamma}z-z|
\]
 is differentiable at $z$, then we have $P_{\Gamma}(z)=z-d_{\Gamma}(z)\nabla d_{\Gamma}(z).$ 
\begin{defn}
Let $u$ be a function defined on $\Gamma$, its \emph{constant-along-normal
extension or the closest point extension} in $T_{\epsilon}$ is defined
by $\bar{u}(z):=u(P_{\Gamma}z),\,\,\,\forall z\in T_{\epsilon}.$ 
\end{defn}
\noindent Next we define equivalence among functions and functionals.
\begin{defn}
Let $v$ be a function defined on $T_{\epsilon}$ and $u$ be a function
defined on $\Gamma$. We say that $v$ is equivalent to $u$ if $v(z)=\overline{u}(z)\,\,\,\forall z\in T_{\epsilon}.$
In this case, we denote $v\equiv u$. Correspondingly, let $\tilde{I}_{T_{\epsilon}}$
and $I_{\Gamma}$ be two integral operators, we define $\tilde{I}_{T_{\epsilon}}\equiv I_{\Gamma}$
if
\begin{equation}
\tilde{I}_{T_{\epsilon}}[v]=I_{\Gamma}[u]\,\,\,\,\text{whenever\,\,\,\,}v\equiv u.\label{def: equivalence}
\end{equation}
\end{defn}
\noindent Naturally, we are interested in answering the questions: 
\begin{question*}
Let $v(z)$ and $u(x)$ be respectively the solutions of the variational
problems involving $\tilde{I}_{T_{\epsilon}}$ and $I_{\Gamma}$.
If $\tilde{I}_{T_{\epsilon}}\equiv I_{\Gamma}$, \textbf{\emph{
\[
\text{\textbf{Is }}v\equiv u?
\]
}}
\end{question*}
Moreover,
\begin{question*}
How stable is the equivalence against perturbation introduced by the
numerical discretization? More precisely, how close is $\frac{\partial v}{\partial n}$
to zero?
\end{question*}

\subsection*{Related work}

\paragraph*{Level set methods or the closest point methods}

\noindent These methods extend the surface PDEs to ones in a narrowband
$T_{\epsilon}$ around the surface $\Gamma$. The proposed work is
motivated by various advantages and disadvantages of these methods.
Of course the aim is to keep the advantages and get rid of the disadvantages.
In the level set methods, e.g. \cite{bertalmio2001variational,dziuk2013finite},
partial derivatives on $\Gamma$, as well as on all nearby parallel
surfaces $\Gamma_{\eta}$   {(the collection of the points
that are $\eta$ distance from $\Gamma$}), are written in the form
of the orthogonal projections of the gradient operator in the Euclidean
space, and thus an equation in $T_{\epsilon}$ is defined formally
replacing the surface gradient $\nabla_{\Gamma}u$ by $(I-\mathbf{n}\otimes\mathbf{n})\nabla v^{\epsilon}$.
As a reminder, $v^{\epsilon}$ is a function defined in $T_{\epsilon}$.
The closest point methods   {for solving parabolic and
elliptic PDEs defined on surfaces}, see e.g. \cite{ruuth_merriman08,macdonald_ruuth09},
assume the constant-along-normal extension of quantities defined on
the surface, and replace surface gradient operator $\nabla_{\Gamma}$
by the gradient operator $\nabla$ defined in the embedding Euclidean
space. Thus the methods typically involve iterations of (a) enforcing
the constraint that solutions are constant-along-normal, $\overline{v^{\epsilon}}(z):=v^{\epsilon}(x)$
for $P_{\Gamma}z=x$, done by extensive interpolation, and (b) solution
of the PDEs in $T_{\epsilon}.$ In either methods, the solution to
a given surface PDE is extracted from the restriction of $v^{\epsilon}$
on $\Gamma$, via interpolation. Compared to the level set methods,
the closest point methods can be applied more easily to manifolds
of different co-dimensions and involve simpler differential operators.
However, these two methodologies seem to be designed exclusively for
solving initial value problems with explicit time stepping.

\paragraph*{Finite element methods on arbitrary surfaces}

A straightforward way to formulate finite element methods to solve
the Laplace-Beltrami problem, $\Delta_{\Gamma}u=f$, on smooth surfaces
$\Gamma=\partial\Omega$ involve: (i) approximate $\Omega$ by a polyhedron
$\Omega_{h}$ such that $\Gamma$ is approximated by $\Gamma_{h,}=\partial\Omega_{h}$
which consists of collection of triangles whose vertices are on $\text{\ensuremath{\Gamma}}$;
(ii) use $\Gamma_{h}$ to parametrize $\Gamma$ locally; (iii) extend
the source term $f$ constant along the normal direction of $\Gamma$
to obtain an equivalent source function $f_{h}$ define on $\Gamma_{h}$;
(iv) solve the weak form on $\Gamma_{h}$ by using standard finite
element methods \cite{MR976234,MR2485433,MR2285862}. Linear finite
elements with piecewise linear triangular approximation $\Gamma_{h}$
is discussed in \cite{MR976234}. For higher order finite element
methods, $\Gamma$ is approximated by $\Gamma_{h}^{k}$, a collection
of images of piecewise interpolating polynomial defined on each triangle
of $\Gamma_{h}$, and use higher order elements defined on $\Gamma_{h}^{k}$
\cite{MR2485433,MR2285862}. These methods require nontrivial lifting
and projection between $\Gamma_{h}^{k}$ and $\Gamma$ when evaluating
integrals in weak form defined on $\Gamma_{h}^{k}.$ Other approaches
approximate surface $\Gamma$ without using nodes on $\Gamma$ directly.
For example, the trace finite element method \cite{MR2551197} or
sharp interface method and narrowband method \cite{MR3249369}, are
designed when the surface $\Gamma$ is given by the zero level set
of a function $\phi$ (not necessary a signed distance function).
The level set function $\phi$ is first approximated by the nodal
interpolant $\phi_{h}$. The approximate surface $\Gamma_{h}$ or
neighborhood of $\Gamma$ are computed from the zero level set or
$\epsilon$-neighborhood of $\phi_{h}$. Other qualities such as normal
of $\Gamma$ and close map on $\Gamma$ are obtained from $\phi_{h}$
approximately. These methods share the same features: (i) alter the
surface $\Gamma$ to an approximate surface $\Gamma_{h}$; (ii) change
the variational form on $\Gamma$ to an equivalent variational form
on $\Gamma_{h}$. However, the implementation of the equivalent weak
form can be very complicated. Recently Olshanskii and Safin proposed
a narrowband unfitted finite element method to solve elliptic PDEs
defined on surfaces \cite{olshanskii2016narrow}. In their work, the
surface PDE is replaced by an equivalent PDE on the narrowband $T_{\epsilon}(\Gamma)$
and an ``equivalent'' PDE is derived from the constant-along-normal
extension of the solution without considering an approximate surface
$\Gamma_{h}.$ The Neumann boundary condition for unfitted element
is treated by a piecewise planar approximation to $\partial T_{\epsilon}$
as proposed in \cite{BARRETT:1984:imanum/4.3.309}. Some local subdivisions
may be needed for the elements that are closed to the boundary. This
work is closest to our formulation. However, as we shall describe
below, our approach does not require adaptive meshing and allows a
very general set of meshing of the ambience space. Both finite difference
and finite element methods can be easily used for discretization.

\paragraph*{Computation of the distance function, the closest point mapping and
constant-along-normal extensions}

The computation of the signed distance functions as well as the closest
point mappings from given level set functions are by now considered
standard component in the level set methods \cite{osher_sethian88},
and can be carried out to high order accuracy in many different ways,
e.g. \cite{sethian1999fast,ahmed2011third,zhang2006high,Redistance_ChengTsai},
where by extending the interface coordinates as constants along interface
normals, $P_{\Gamma}$ can be computed easily to fourth order in the
grid spacing. Once an accurate distance function is computed on the
grid, its gradient can be computed by standard finite differencing
or by more accurate but wider WENO stencils \cite{jiang2000weighted}.
If the surfaces are given initially by explicitly parameterized patches,
fast algorithms such as the one proposed in \cite{Tsai:JCP2002} in
combination with KD-Tree can be used. 

\paragraph*{The proposed framework}

The proposed framework contains a theoretically sound formulation
that admits unique solutions which are constant-along-normal, i.e.
$v^{\epsilon}(z)=v(P_{\Gamma}z)$ for all $z\in T_{\epsilon}$, \emph{without
enforcing extra constraints}. The formulation will include relatively
simple differential operators and a theory on suitable and easy to
implement boundary conditions. As a result, the proposed method can
easily be applied to solve problems involving curves and surfaces
in three dimensions, integral constraints and eigenvalues \textemdash{}
all of which pose some level of difficulty for the other methods.
Some of the operators in the equations resulting from the proposed
framework will resemble those used in \cite{greer2006improvement,olshanskii2016narrow,vogl2016curvature}.
However, the way the equations are derived, the boundary conditions,
and the numerical methods differ significantly. We also propose a
method to extend the wave equation on surface in volumetric form and
treat the Neumann boundary condition in an easy way. The instability
due to the approximation of boundary condition can be corrected by
modifying the extended PDE.

In summary, the proposed method provides a unifying approach to tackle
a wide range of problems defined on surfaces or curves in three dimensions.
The reasons for considering the proposed framework are simple: the
extended problems are easier to solve computationally, in particular
for applications where the surfaces and curves are given non-parametrically
and are changing dynamically, and/or some global information about
the operators such as the eigenvalues and eigenfunctions needs to
be computed. The variational integrals used in the formulation, discussed
in Section 2, provide a natural and systematic way of investigating
important issues on boundary conditions, regularization and compatibility.
The proposed method also opens up possibility of direct numerical
minimization methods for problems involving more nonlinear and degenerate
surface equations. However, this direction will be investigated in
a future paper. 

\subsection*{Notations for finite difference discretization}

In this paper, we shall use the following notations and setup for
discussion involving finite difference discretization of the proposed
equations and boundary conditions. Approximation of an unknown function
$u$ will be constructed on the uniform Cartesian grid \textbf{$h\mathbb{Z}^{2}$}
for some $h>0$, and the typical notation $u_{i,j}$ for approximation
of $u(ih,jh)$. For evolution problems, $u_{i,j}^{n}$ will denote
the approximation for $u(ih,jh,n\Delta t)$ where $\Delta t$ is the
step size and $n$ is a positive integer. We shall also use the standard
notations for finite difference operators acting on the grid function
$u_{i,j}$:
\[
D_{\pm}^{x}u_{i,j}=\pm\frac{u_{i\pm1,j}-u_{i,j}}{h},D_{\pm}^{y}u_{i,j}=\pm\frac{u_{i,j\pm1}-u_{i,j}}{h}.
\]
 These finite difference operators, $D_{+}^{x},D_{-}^{x},D_{+}^{y}$
and $D_{-}^{y}$, are used to construct finite difference approximations
of higher order partial derivatives of $u$. Polynomials of these
four operators correspond to applying the operators recursively to
the grid function as defined by the polynomials: e.g. 
\[
D_{+}^{x}D_{-}^{x}u_{i,j}=D_{+}^{x}\frac{u_{i,j}-u_{i-1,j}}{h}=\frac{u_{i+1,j}-2u_{i,j}+u_{i-1,j}}{h^{2}}
\]
 is a standard approximation of $u_{xx}(ih,jh)$, and 
\begin{align*}
(D_{+}^{x}D_{-}^{x})^{2}u_{i,j} & =D_{+}^{x}D_{-}^{x}(\frac{u_{i+1,j}-2u_{i,j}+u_{i-1,j}}{h^{2}})\\
 & =\frac{1}{h^{2}}D_{+}^{x}D_{-}^{x}u_{i+1,j}-\frac{2}{h^{2}}D_{+}^{x}D_{-}^{x}u_{i,j}+\frac{1}{h^{2}}D_{+}^{x}D_{-}^{x}u_{i-1,j}
\end{align*}
 is an approximation of $u_{xxxx}(ih,jh).$ 

\subsection*{Layout of the paper}

This paper is organized as follows. In Section 2, we describe the
proposed extensions and boundary closure for discretization of the
extended problems on Cartesian grids or more general meshes. In Section
3, we consider minimization of strictly convex functionals which lead
to elliptic equations. In Section 4, we study least-action principles
that lead to hyperbolic problems. In both Sections 3 and 4, we study
the stability of the proposed extensions under the presence of perturbation
to the boundary (and initial) conditions. A brief summary of our findings
is presented in Section 5. In the Appendix, we provide some details
on the derivation of the operators used in the paper. 

\section{  {The proposed framework}\label{sec:The-Research-Program}}

  {In our framework, the closest point projection $P_{\Gamma}$
is used to define extension of integrants as well as to provide a
change of variables for integration in the narrowband $T_{\epsilon}$.
In extending a variational problem defined on $\Gamma$ to one defined
in $T_{\epsilon}$,} one needs to deal with the additional degrees
of freedom in the co-dimensions of $\Gamma$: depending on the nature
of $I_{\Gamma}$, one may need to do nothing, impose restrictions,
adding a regularizing term specific to the additional dimension(s)
or convexifying the resulting functional.   {Additionally,
one needs to develop efficient stable numerical methods for enforcing
suitable boundary conditions.}  

\subsection{  {Extensions of functionals defined on closed surfaces
of co-dimension one} \label{subsec:Closed-surfaces-embedded}}

We start with the general procedure proposed in \cite{KTT} for integration
on closed hypersurfaces. The procedure starts by rewriting the integral
on $\Gamma$ using nearby parallel curves or surfaces. It is useful
to have the following notation: 
\begin{defn}
Let $d_{\Gamma}$ be a signed distance function to $\Gamma=\partial\Omega$;
i.e. $d_{\Gamma}(w)<0$ for any $w\in\Omega$ and $d_{\Gamma}(w)>0$
for any $w\in\bar{\Omega}^{c}.$ The ``parallel surface'' which
is of $\eta$ distance from $\Gamma$ is denoted by 
\[
\Gamma_{\eta}:=\{y\in\mathbb{R}^{d}:d_{\Gamma}(y)=\eta\}.
\]
The closest point projection $P_{\Gamma}$ is a diffeomorphism between
$\Gamma$ and $\Gamma_{\eta}$ for any $\eta\in[0,\kappa_{\infty}^{-1})$
where $\kappa_{\infty}$ is an upper bound of the curvatures of $\Gamma$. 
\end{defn}
\noindent The integrals on the parallel curves or surfaces, indexed
by the distance to $\Gamma$, are then averaged using a kernel $K_{\epsilon}\in C([-\epsilon,\epsilon])$
that satisfies $\int_{-\epsilon}^{\epsilon}K_{\epsilon}(\eta)\,d\eta=1$.
We summarize the two steps as follows:

\begin{equation}
\int_{\Gamma}L(x)dS(x)=\int_{\Gamma_{\eta}}L(P_{\Gamma}y)J(y)dS(y)=\int_{-\epsilon}^{\epsilon}K_{\epsilon}(\eta)\int_{\Gamma_{\eta}}L(P_{\Gamma}y)J(y)dS(y)\,d\eta.\label{eq:integration-by-layers}
\end{equation}
Here, $J$ is the Jacobian that comes from the change of variables
$z\mapsto P_{\Gamma}z\in\Gamma$. For curves embedded in $\mathbb{R}^{2}$,
$J(z)$ is given by

\[
J(z)=1-d_{\Gamma}(z)\kappa(z),
\]
where $\kappa$ is the curvature of the parallel curve $\Gamma_{\eta}$
that passes through $z$; i.e. $\eta=d_{\Gamma}(z)$. See Figure~\ref{fig: J-and-BC}
for an illustration. Similarly, for surfaces embedded in $\mathbb{R}^{3}$,
$J(z)$ is give by

\begin{equation}
J(z)=\prod_{j=1}^{2}[1-d_{\Gamma}(z)\kappa_{j}(z)]=\sigma_{1}(z)\sigma_{2}(z),\label{eq:The-Jacobian}
\end{equation}
where $\kappa_{1},\kappa_{2}$ are the two principle curvatures of
the parallel surface $\Gamma_{\eta}$ that passes through $z$. In
\cite{kublik2016integration}, we show that $\sigma_{j}$ are the
two largest singular values of $DP_{\Gamma}$, the derivative of the
closest point projection operator. Thus $J(z)$ can be easily computed
by finite differencing $P_{\Gamma}z$ and singular value decomposition. 

Finally, by the coarea formula, see e.g. \cite{Evans-Gariepy}, the
double integral on the right hand side of \eqref{eq:integration-by-layers}
equals to a Lebesque integral in $T_{\epsilon}$, and thus
\[
\int_{\Gamma}L(x)dS(x)=\int_{T_{\epsilon}}L(P_{\Gamma}z)\,K_{\epsilon}(d_{\Gamma}(z))J(z)\,dz.
\]
If $L$ depends on a function $u:\Gamma\mapsto\mathbb{R}$, there
is no difficulty in applying the same extension procedure: 
\[
L(x,u(x))\longrightarrow L(P_{\Gamma}z,u(P_{\Gamma}z)).
\]

Now if $L$ depends on $\nabla_{\Gamma}u$, the gradient of a function
$u$ on $\Gamma$, one could apply the same procedure as above \emph{if
there is a convenient formula for the constant-along-normal extension
of $\nabla_{\Gamma}u(x).$ }  {Naturally, we would like
to relate the normal extension of the surface gradient $\overline{\nabla_{\Gamma}u}(z)$
to the gradient of the normal extension $\nabla_{z}\overline{u}(z)$
for $z\in T_{\epsilon}.$} 

We consider a simple motivating example involving concentric circles.
\begin{example}
Let $\Gamma_{\eta}$ be the circle, centered at the origin, with radius
$r_{0}+\eta$ for some $r_{0}>0$. Let $\bar{u}$ be a $C^{1}$ function
on $\mathbb{R}^{2}\setminus\{(0,0)\}$. Consider polar coordinates
$(r,\theta)$. On $\Gamma_{\eta}$, the derivative of $\bar{u}$ with
respect to the arc length is related to that with respect to the angle
$\theta$ by 
\[
\frac{\partial\bar{u}}{\partial s}=\frac{\partial\bar{u}}{\partial\theta}\frac{1}{r_{0}+\eta}.
\]
Thus the curve derivative of $\bar{u}$ on $\Gamma$ is equivalent
to the that on $\Gamma_{\eta}$, weighted properly by $J^{-1}:$ 
\[
\frac{\partial\bar{u}}{\partial s}|_{\Gamma}=\frac{1}{r_{0}}\frac{\partial\bar{u}}{\partial\theta}=\frac{r_{0}+\eta}{r_{0}}\frac{\partial\bar{u}}{\partial s}|_{\Gamma_{\eta}}=J^{-1}\frac{\partial\bar{u}}{\partial s}|_{\Gamma_{\eta}}.
\]
We can furthermore rewrite this equality into a general formula involving
$\nabla\bar{u}$, the gradient of $\bar{u}$ in $\mathbb{R}^{2}$,
and tangential and normal vectors, $\mathbf{t}$ and \textbf{n}, of
the concentric circles: 
\[
\frac{\partial\bar{u}}{\partial s}|_{\Gamma}=J^{-1}(\nabla\bar{u}\cdot\mathbf{t})\mathbf{t}+\mu(\nabla\bar{u}\cdot\mathbf{n})\mathbf{n}=\left(J^{-1}\mathbf{t}\otimes\mathbf{t}+\mu\mathbf{n}\otimes\mathbf{n}\right)\nabla\bar{u},
\]
for any $\mu.$ $\qed$
\end{example}
It can be shown easily that the surface gradient $\nabla_{\Gamma}u$,
when embedded as a vector in $\mathbb{R}^{3}$, satisfies
\begin{equation}
\overline{\nabla_{\Gamma}u}(z)=A(z;\mu)\nabla_{z}\overline{u}(z),\label{eq:gradient-projection-tensor}
\end{equation}
where $A(z;\mu)=A_{0}(z)+\mu A_{1}(z),$
\begin{align*}
A_{0}(z):= & \sigma_{1}^{-1}\mathbf{t}_{1}\otimes\mathbf{t}_{1}+\sigma_{2}^{-1}\mathbf{t}_{2}\otimes\mathbf{t}_{2},\\
A_{1}(z):= & \mathbf{n}\otimes\mathbf{n},
\end{align*}
and $\mathbf{t}_{1}$, $\mathbf{t}_{2}$ are the two orthonormal tangent
vectors corresponding to the directions that yield the principle curvatures
of $\Gamma$, $\mathbf{n}$ is the unit normal vector, $\mu$ can
be any real number, and $\sigma_{1},$$\sigma_{2}$ are defined in
\eqref{eq:The-Jacobian}. Thus $A_{0}$ corresponds to a weighted
orthogonal projection onto the tangent plane and $A_{1}$ the projection
along the normal of the parallel surface passing through $z.$ See
Appendix for detail. The projection $\mu A_{1}\nabla_{z}\overline{u}(z)$
is always zero since $\overline{u}(z)$ is constant long the normal
of $\Gamma$. The constituents of $A(z;\mu)$ can be easily computed
from either the eigenvalues and eigenvectors of of $DP_{\Gamma}$
or $D^{2}d_{\Gamma}$; see \cite{kublik2016integration}. 

  {In this paper, we consider extensions of $I_{\Gamma}[u]$
in the form:
\begin{equation}
\tilde{I}_{T_{\epsilon}}[v]:=\int_{T_{\epsilon}}\underbrace{L(P_{\Gamma}z,v(z),A_{0}\nabla v(z)+\mu A_{1}\nabla v(z))\,K_{\epsilon}(d_{\Gamma}(z))J(z)}_{\bar{L}(z,v,\nabla v)}\,dz,\label{eq:real-def-of-J_Teps}
\end{equation}
where $\mu A_{1}\nabla v$ is the additional term corresponding to
the extra degree of freedom in the co-dimension. If $I_{\Gamma}[u]$
is convex, the term $\mu A_{1}\nabla v$ should }  {\emph{convexify}}  {{}
the extended problem to ensure existence and uniqueness of the minimizer,
and together with the natural boundary condition at $\partial T_{\epsilon}$
it should ensure that the minimizer satisfies Assumption~\ref{assu:dv/dn=00003D0}.
We refer to this additional term as the }  {\emph{regularization
term }}  {in $\tilde{I}_{T_{\epsilon}}$. The degree
of freedom from the co-dimension needs to be treated differently for
the saddle-point problem discussed in Section \ref{sec:Least-action-principles}. }

Also in this paper, we shall consider $K_{\epsilon}=\frac{1}{2\epsilon}\chi_{[-\epsilon,\epsilon]}$
for simplicity. The operator $\tilde{I}_{T_{\epsilon}}$ is equivalent
to the original operator $I_{\Gamma}$. Instead of finding the minimizer
of the energy $I_{\Gamma}[\cdot]$ in a Banach space of functions
defined on $\Gamma$, we look for the minimizer of $\tilde{I}_{T_{\epsilon}}[\cdot]$
in a suitable Banach space of functions defined in $T_{\epsilon}$.
The derivatives involved in $\tilde{I}_{T_{\epsilon}}[v]$ are in
the Eulerian coordinates and therefore are easier to compute numerically.
The geometry information of $\Gamma$ is embedded in the coefficient
$A(z;\mu)$ and the boundary of $T_{\epsilon}$.\textcolor{blue}{{}
}We have been keeping the discussion on the relevant function spaces
at minimum.   {In general, the choice of the function
spaces for the extended problems depends on the specific problem class.}

\begin{figure}
\begin{centering}
\includegraphics{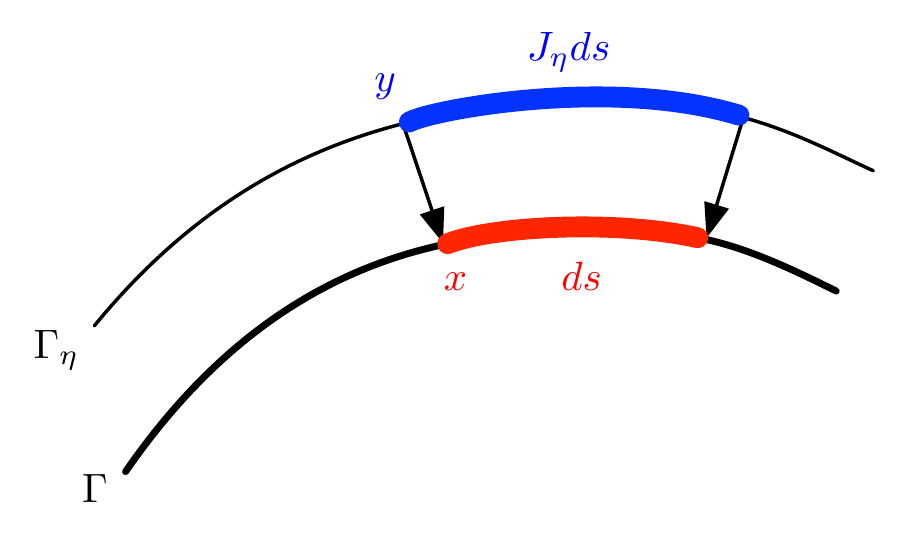}
\par\end{centering}
\caption{Parameterization of $\Gamma$ by a nearby parallel surface that is
$\eta$ distance away. \label{fig: J-and-BC}}
\end{figure}

\subsection{  {Extensions of functionals defined on closed surfaces
of co-dimension two}}

In this paper, we only consider the extension for functionals defined
on closed curves embedded in $\mathbb{R}^{3}.$ The procedure for
integrating hypersurfaces is generalized in \cite{kublik2016integration}
to this case, with the the Jacobian 
\begin{equation}
J(z)=\frac{\sigma_{1}(z)}{2\pi d_{\Gamma}(z)}.\label{eq:Jacobian-curve-in-R3}
\end{equation}
Notice that the division by $d_{\Gamma}$ accounts for the fact that
we are approximating a line integral (integration along $\Gamma$)
by a family of surface integrals (integration on $\Gamma_{\eta},$
$\eta\in(0,\epsilon]$). We may remove this singularity by taking
$K_{\epsilon}(\eta)=2\pi\eta\tilde{K}_{\epsilon}(\eta),$ $\tilde{K}_{\epsilon}\in C^{p}([0,\epsilon])$
and obtain 
\[
\int_{\Gamma}u(x)dS(x)=\frac{1}{Z_{0}}\int_{0}^{\epsilon}K_{\epsilon}(\eta)\int_{\Gamma_{\eta}}u(P_{\Gamma}z)\frac{\sigma_{1}(z)}{2\pi\eta}dS(z)\,d\eta=\frac{1}{Z_{0}}\int_{T_{\epsilon}}u(P_{\Gamma}z)\tilde{K}_{\epsilon}(d_{\Gamma}(z))\sigma_{1}(z)\,dz,
\]
 where $Z_{0}=\int_{0}^{\epsilon}K_{\epsilon}(\eta)d\eta$ is the
normalizing factor. 

Finally, the projection tensor in \eqref{eq:gradient-projection-tensor}
has two degrees of freedom, and takes the from:

\[
A(z;\mu_{1},\mu_{2}):=\left(\sigma_{1}^{-1}\mathbf{t}_{1}\otimes\mathbf{t}_{1}+\mu_{1}\mathbf{n}_{1}\otimes\mathbf{n}_{1}+\mu_{2}\mathbf{n}_{2}\otimes\mathbf{n}_{2}\right),
\]
 where $\mu_{1},\mu_{2}$ are two parameters corresponding to regularization
in the co-dimensions to the curve, and $\sigma_{1}$ is the largest
singular value of $DP_{\Gamma}(z).$ 

\subsection{Equivalence between the minimizers of surface energy $I_{\Gamma}$
and its extension $\tilde{I}_{T_{\epsilon}}$}

To justify our proposed method, we must demonstrate that $v\equiv u$
whenever $I_{\Gamma}\equiv\tilde{I}_{T_{\epsilon}}$. For simplicity,
we only discuss the strictly convex energy

\[
I_{\Gamma}[u]=\int_{\Gamma}L(x,u(x),\nabla_{\Gamma}u(x))dS(x).
\]
Saddle point type energy is discussed in Section \ref{sec:Least-action-principles}. 
\begin{thm}
  {\label{thm:Equivalent}Suppose $\Gamma$ is a $C^{2}$-surface
in $\mathbb{R}^{3}$ and $L(x,q,\mathbf{p}):\Gamma\times\mathbb{R}\times\mathbb{R}^{3}\rightarrow\mathbb{R}$
is twice continuously differentiable and satisfies $L_{\mathbf{p}}(x,q,\mathbf{t})\cdot\mathbf{n}(x)=0$
for and vector $\mathbf{t}$tangent to $\Gamma$ at $x$. Let $u$
be a $C^{2}$-solution of the Euler-Lagrange equation for surface
energy $I_{\Gamma}$, $\delta I_{\Gamma}[u]=0$. Then its normal extension
$\overline{u}:=u(P_{\Gamma}\cdot)$ is a solution of the Euler-Lagrange
equation for extended energy $\tilde{I}_{T_{\epsilon}}$; i.e. $\delta\tilde{I}_{T_{\epsilon}}[\bar{u}]=0$,
and $\bar{u}$ satisfies the Neumann boundary condition on $\partial T_{\epsilon}$.}
\end{thm}
\begin{proof}
Since $u$ is a solution of the Euler-Lagrange equation $\delta I_{\Gamma}[u]=0$,
$u$ satisfies
\begin{equation}
-\nabla_{\Gamma}\cdot L_{{\bf p}}(x,u,\nabla_{\Gamma}u)+L_{q}(x,u,\nabla_{\Gamma}u)=0\,\,\forall x\in\Gamma.
\end{equation}
By Theorem \ref{thm:A1} in the Appendix, we have
\begin{equation}
\begin{split}\begin{split}\delta\tilde{I}_{T_{\epsilon}}[\bar{u}]= & -\nabla\cdot(JAL_{{\bf p}}(P_{\Gamma}z,\bar{u},A\nabla\bar{u}))+JL_{q}(P_{\Gamma}z,\bar{u},A\nabla\bar{u})\\
= & J(-J^{-1}\nabla\cdot(JAL_{{\bf p}}(P_{\Gamma}z,\overline{u},\overline{\nabla_{\Gamma}u}))+L_{q}(P_{\Gamma}z,\overline{u},\overline{\nabla_{\Gamma}u}))\\
= & J(-J^{-1}\nabla(\cdot JA\overline{L_{{\bf p}}(x,u,\nabla_{\Gamma}u)})+\overline{L_{q}(x,u,\nabla_{\Gamma}u)})\\
= & J(-\overline{\nabla_{\Gamma}\cdot L_{{\bf p}}(x,u,\nabla_{\Gamma}u)}+\overline{L_{q}(x,u,\nabla_{\Gamma}u)})=0
\end{split}
\end{split}
\end{equation}
Therefore $\bar{u}$ satisfies the Euler-Lagrange equation for the
extended energy $\tilde{I}_{T_{\epsilon}}$. It is obvious that $\frac{\partial\bar{u}}{\partial n}=0$
on $\partial T_{\epsilon}$ since it is true inside $T_{\epsilon}$
as well.
\end{proof}
\begin{cor}
If the normal extension of a function $u$ on $\Gamma$ satisfies
$\delta\tilde{I}_{T_{\epsilon}}[\bar{u}]=0$, then $u$ is a solution
of $\delta I_{\Gamma}[u]=0$. 
\end{cor}
\begin{proof}
This follows from the proof in Theorem \ref{thm:Equivalent} directly.
\end{proof}
\begin{cor}
If the Euler-Lagrange equation $\delta I_{\Gamma}[u]=0$ is solvable
and $\delta\tilde{I}_{T_{\epsilon}}[v]=0$ with the Neumann boundary
condition has the unique solution $v$, then $v$ is constant-along-normal
and $v_{\Gamma}$, the restriction of $v$ on $\Gamma$, is the unique
solution of $\delta I_{\Gamma}[u]=0$.
\end{cor}
\begin{proof}
Let $u$ be a solution of $\delta I_{\Gamma}[u]=0$. Then by Theorem
\ref{thm:Equivalent}, $\overline{u}$ is a solution for $\delta\tilde{I}_{\Gamma}[v]=0$
with Neumann boundary condition. By uniqueness, we conclude that $v=\overline{u}$.
Hence $v$ is constant-along-normal and $\delta I_{\Gamma}[u]=0$
has exactly one solution as well.
\end{proof}
\begin{example}
The Laplace-Beltrami equation on $\Gamma$ reads as
\begin{equation}
-\Delta_{\Gamma}u=f\,\text{on}\,\Gamma,\label{eq:LB}
\end{equation}
and its extension in $T_{\epsilon}$ is

\begin{equation}
-\nabla\cdot(JA^{2}\nabla v)=\overline{f}J\,\,\text{on}\,\,T_{\epsilon},\,\,\frac{\partial v}{\partial n}=0\,\,\text{on}\,\,\partial T_{\epsilon}.\label{eq:LB_extended}
\end{equation}
Since $\int_{T_{\epsilon}}\overline{f}\,Jdz=\int_{\Gamma}f\,dS$,
the solvability for both equations are the same. We convexify the
extended energy by using a positive constant $\mu$ in $A(z;\mu)$.
The solution is unique up to an additive constant for both the original
and extended problems. Therefore a solution of \eqref{eq:LB_extended}
is a normal extension of some solution of \eqref{eq:LB}. In the weak
formulations, the solution space for \eqref{eq:LB} is usually chosen
to be $H^{1}(\Gamma)$ with zero mean, and the solution space for
\eqref{eq:LB_extended} is usually chosen to be $H^{1}(T_{\epsilon})$
with zero mean.
\end{example}

\subsection{Boundary closure\label{subsec:proposed-BC}}

In the previous two subsections, we derive the integral operators
$\tilde{I}_{T_{\epsilon}}$ which are equivalent to $I_{\Gamma}$.
The integral $\tilde{I}_{T_{\epsilon}}[v]$ or it's variational derivative
and the accompanying boundary conditions will be discretized. In this
paper, we consider meshes that do not follow the geometry, in particular
Cartesian grids, and unavoidably, we need to address the issue of
discretization near the boundary $\partial T_{\epsilon}$. More specifically,
near $\partial T_{\epsilon}$ numerical approximation of partial derivatives
of $v$ will involve ``ghost-values'' on some ``ghost nodes''
lying outside of $T_{\epsilon}$ In this section, we describe a boundary
closure procedure that works specifically to our setup. 
\begin{assumption}
\label{assu:dv/dn=00003D0} The solution to the extended problem,
$\tilde{I}_{T_{\epsilon}}[v]=\bar{f}$ or $\delta\tilde{I}_{T_{\epsilon}}[v]=0$
with suitable boundary conditions, satisfies
\[
\frac{\partial v}{\partial n}:=\mathbf{n}\cdot\nabla v\equiv0\,\,\,\text{in}\,\,\,T_{\epsilon},
\]
 where $\mathbf{n}(z)=\nabla d_{\Gamma}(z)$ is the normal vector
of $\Gamma$ at $P_{\Gamma}z.$
\end{assumption}
We illustrate the essential idea of the proposed boundary closure
by finite difference methods in the left subfigure of Figure~\ref{fig:domain_ellipse}.
In the subfigure, the red nodes, also called as ghost nodes, are outside
of $T_{\epsilon}$ but they are needed as part of the finite difference
stencil for some blue nodes inside $T_{\epsilon}.$ We need to assign
values to these red nodes in order to close the discretized system
that contains only the interior blue nodes. We propose that:
\begin{itemize}
\item Each ghost node will be projected along the respective normal line
back into a cell close to $\partial T_{\epsilon}$. In the figure,
one of such cells is outlined in blue. 
\item An interpolation is performed using the grid values surrounding the
blue cell. The order of interpolation must be higher or equal to the
order used in discretization of the PDE on $T_{\epsilon}.$ The higher
order interpolation the wider stencil as well as the wider narrowband.
\end{itemize}
Figure~\ref{fig:domain_ellipse} shows an illustration of the boundary
closure procedure that we implemented for an ellipse. 

\begin{figure}
\begin{centering}
\includegraphics{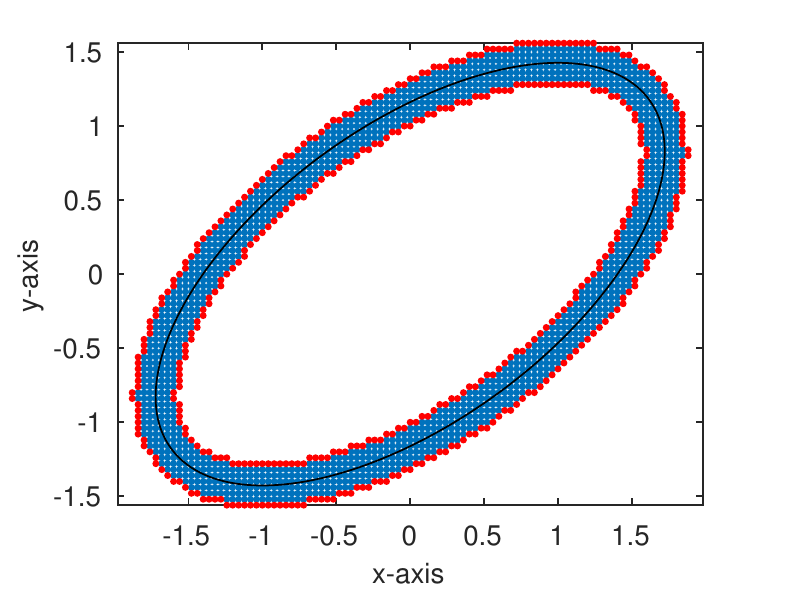}\includegraphics{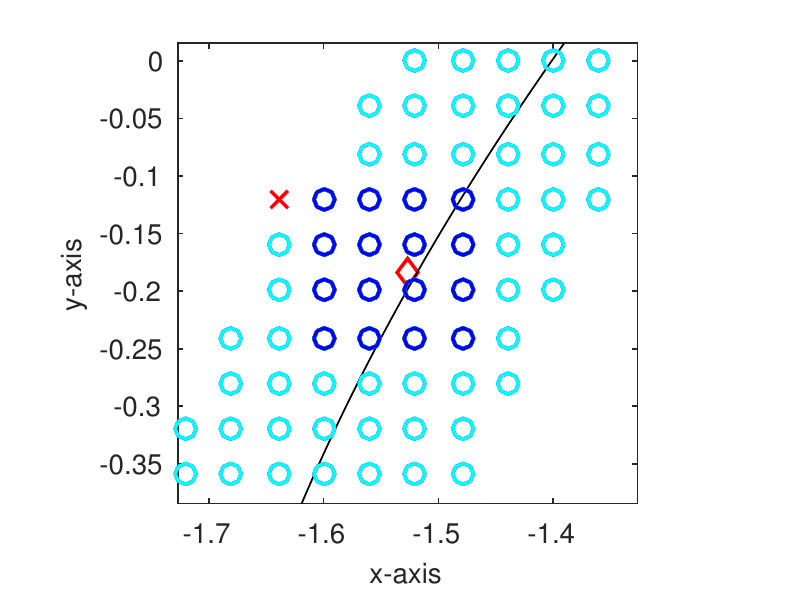}
\par\end{centering}
\caption{\label{fig:domain_ellipse}(Left) Illustration of computational domain
when $\epsilon=3h$ and $h=1/25$. Black curve is the ellipse $\Gamma$,
blue nodes are inner grids and red nodes are ghost nodes. (Right)
The red x represents a ghost nodes and the red diamond represents
the projected node we use to compute the value on the ghost node.
  {Notice that, unlike the close point methods, ghost
nodes do not need to be projected back to $\Gamma$. }The nearby 16
dark blue nodes are used to interpolate the projected node. }
\end{figure}

\subsubsection*{Super-convergence of the boundary closure}

We briefly explain why the proposed boundary closure can replace zero
Neumann boundary conditions at $\partial T_{\epsilon}$. For simplicity,
we demonstrate the two dimensional case. Let $(x_{i},y_{j})$ be a
ghost node lying outside of $T_{\epsilon}$ and $(x_{i}^{\alpha},y_{j}^{\alpha})=(x_{i},y_{j})-\alpha\mathbf{n}_{i,j}$
be the projected node inside $T_{\epsilon}$ for some $\alpha>0$
of order $h$. Let $v_{i,j}^{\alpha}:=v(x_{i}^{\alpha},y_{j}^{\alpha})$
, then
\[
v_{i,j}=v_{i,j}^{\alpha}\iff\frac{v_{i,j}-v_{i,j}^{\alpha}}{\alpha}=0,
\]
which may be considered as a first order approximation of the zero
Neumann boundary condition at $(x_{i}^{*},y_{j}^{*}):=(x_{i},y_{j})-\theta\mathbf{n}_{i,j}\in\partial T_{\epsilon}$
for some $0\le\theta<\alpha.$ We extend the solution constant-along-normal
further from $\partial T_{\epsilon}$ and assume that the extended
$v$ is still smooth (for this we need to require that $\epsilon+2h<\kappa_{\infty}$).
This means that 
\[
\frac{\partial v}{\partial\mathbf{n}}(x_{i}^{*},y_{j}^{*})=\frac{v_{i,j}-v_{i,j}^{\alpha}}{\alpha}+\frac{1}{2}(\alpha-2\theta)\frac{\partial^{2}v}{\partial\mathbf{n}^{2}}(x_{i}^{*},y_{j}^{*})+\mathcal{O}(h^{2}).
\]
 However, if $v$ satisfies Assumption~\ref{assu:dv/dn=00003D0},
$\frac{\partial^{p}v}{\partial\mathbf{n}^{p}}(x_{i}^{*},y_{j}^{*})=0$,
$p=1,2,\cdots$, we see that there is no error in the boundary closure. 

In the procedure described above, $v_{i,j}^{\alpha}$ is then approximated
to sufficient order of accuracy by interpolation. This will perturb
the boundary condition to be of the form
\[
\frac{\partial v}{\partial\mathbf{n}}\approx h^{k-1}\frac{\partial^{k}v}{\partial\mathbf{t}^{k}}+...
\]
for some positive integer $k\geq1$. It is not hard to argue that
the proposed boundary closure yields convergent approximations of
the solution to the given PDE with zero Neumann boundary conditions.
But it is not so trivial to show that the convergence is high order
for pure boundary value problems or evolution problems discretized
by implicit schemes. 

In particular, if we use bi-cubic interpolation centered at grid cell
containing $(x_{i}^{\alpha},y_{j}^{\alpha})$  to approximate $v_{i,j}^{\alpha}$,
simple calculations shows that the leading error in the approximation
of $v_{i,j}^{\alpha}$ is of the form $h^{4}(c_{1}\frac{\partial^{4}u}{\partial x^{4}}+c_{2}\frac{\partial^{4}u}{\partial y^{4}})$
with positive coefficients $c_{1}$ and $c_{2}$. After a change of
coordinates, and assuming that $v$ satisfies Assumption~\ref{assu:dv/dn=00003D0},
the perturbed the boundary condition is given by
\begin{equation}
\frac{\partial v}{\partial\mathbf{n}}=ch^{3}\frac{\partial^{4}v}{\partial\mathbf{t}^{4}}+\text{(lower order in \ensuremath{h} terms)}\label{eq:cubic_perturb_bd}
\end{equation}
with some positive coefficient $c$. In Section \ref{sub:Laplace-Beltrami},
we will see that the sign of $c$ plays an important role in stability
of the proposed method. 

\subsubsection*{Depth of projected points}

We briefly discuss how to choose a suitable projection ``depth''
$\alpha>0$. Let ${\bf x}$ be a ghost node and ${\bf x}^{\alpha}={\bf x}-\alpha{\bf n}$
be the projected node inside $T_{\epsilon}$. When $v$ is constant-along-normal,
any point along the normal direction has the same value of $v$. However,
in numerical experiments, we need to choose a suitable projected node.
The necessary condition for choosing the projected node is that the
nearby surrounding nodes used to interpolate the value of $v$ must
be inside $T_{\epsilon}$. In our numerical simulations, we use cubic
interpolating polynomials in a dimension-by-dimension fashion to approximate
$v({\bf x}^{\alpha})$. Hence it requires $4^{3}$ nearby points for
3 dimensional cases ($4^{2}$ nearby points for 2 dimensional case).
See Figure \ref{fig:domain_ellipse} for illustration for a two dimensional
case. We suggest that the projection depth be as shallow as possible.
Therefore in our numerical simulation, we choose $\alpha=d_{\Gamma}(x)-(\epsilon-2\sqrt{d}h),$
  {which guarantees that the interpolation stencil stays
inside $T_{\epsilon}$ if $\epsilon=mh$, $h<1/m\kappa_{\infty}$
and $m>2.$ }Here $\kappa_{\infty}$ is an upper bound of the curvatures
of $\Gamma$.  {{} }We discuss the effect of different
$\alpha$ in the examples in Section \ref{subsec:Sensitivity-of-boundary}.

\section{Minimization of convex energies\label{sec:Energy-minimization}}

In this section, we shall first study a simple model problem defined
on a flat torus. Such study will reveal many essential properties
of the proposed algorithms, in particular, the stability of the extended
problem, the consistency of the extended gradient flow. 

\subsection{Study of a model problem \label{sub:Laplace-Beltrami} }

Let $\Gamma$ to be the unit interval on the $x$-axis with periodic
boundary conditions at $x=0$ and $1$, and $T_{\epsilon}=[0,1)\times(-\epsilon,\epsilon).$
For any $f\in C([0,1))$ with $f(0)=f(1)$ and $c>0$, we consider
the energy
\[
I_{\Gamma}[u]:=\int_{\Gamma}\frac{1}{2}(|\nabla_{\Gamma}u|^{2}+cu^{2})-fu\,dS=\int_{\Gamma}\frac{1}{2}(|u_{x}|^{2}+cu^{2})-fu\,dx,
\]
and its extension (with $R(v_{y})=v_{y}^{2}$) 
\[
\tilde{I}_{T_{\epsilon}}[v]:=\int_{-\epsilon}^{\epsilon}K_{\epsilon}(y)\left(\int_{0}^{1}\frac{1}{2}(|v_{x}|^{2}+\mu v_{y}^{2}+cv^{2})-\bar{f}v\,dx\right)\,dy.
\]
The Euler-Lagrange equation of $I_{\Gamma}$ is 
\begin{equation}
-u_{xx}+cu=f,\,\,\,0<x<1,\,\,\,\text{and}\,\,u(0)=u(1),u'(0)=u'(1).\label{eq:model-in-the-strip}
\end{equation}
The Euler-Lagrange equation of $\tilde{I}_{T_{\epsilon}}$ and the
natural boundary conditions for $v\in C^{2}(T_{\epsilon})$ are
\begin{equation}
-v_{xx}-\mu v_{yy}+cv=\bar{f}\,\,\,\text{in}\,T_{\epsilon},\,\,\,\,\,v(0,y)=v(1,y),v_{x}(0,y)=v_{x}(1,y)\label{eq:model-extended-Laplace-Beltrami}
\end{equation}
\begin{equation}
v_{y}(x,\pm\epsilon)=0,\,\,\,0\le x<1.\label{eq:model-problem:v_y=00003D0}
\end{equation}
 Here, we make a few observations for the case of $\mu>0$:
\begin{itemize}
\item $\tilde{I}_{T_{\epsilon}}[v]$ \emph{is strictly convex, }and the
Euler-Lagrange equation \eqref{eq:model-extended-Laplace-Beltrami}
of $\tilde{I}_{T_{\epsilon}}$ is uniformly elliptic, and has a unique
solution in $C^{2}(T_{\epsilon})$ with the Neumann boundary condition~\eqref{eq:model-problem:v_y=00003D0}. 
\item The $y$-derivative of the minimizer of $\tilde{I}_{T_{\epsilon}}$
has to be 0, for otherwise, $\tilde{I}_{T_{\epsilon}}$ can still
take smaller values by diminishing $v_{y}^{2}$. 
\item For any function $\tilde{v}$ satisfying $\tilde{v}_{y}\equiv0$ in
$T_{\epsilon},$ $\tilde{I}_{T_{\epsilon}}[\tilde{v}]=I_{\Gamma}[\tilde{v}(\cdot,0)].$
Consequently, we can argue that $u(x):=v(x,0)$ is the minimizer of
$I_{\Gamma}$, the solution of the corresponding Laplace-Beltrami
equation. 
\end{itemize}

\subsubsection{Stability of the constant-along-normal solutions}

We investigate this issue in three regards: (i) whether the interpolation
used in the boundary closure introduces unstable solutions, and (ii)
for time dependent problem, whether perturbation to the initial conditions
will be amplified in time, and (iii) the effect of the inhomogeneous
term not being perfectly constant-along-normal. 

We study these issues with Example~\eqref{sub:Laplace-Beltrami},
involving \eqref{eq:model-extended-Laplace-Beltrami}, with $c>0$.
\begin{equation}
\frac{\partial v}{\partial y}=\begin{cases}
\alpha\frac{\partial^{k}v}{\partial x^{k}}, & y=\epsilon,\\
0, & y=-\epsilon.
\end{cases}\label{eq:perturbed-Neumann}
\end{equation}
  {Due to linearity of the problem, the solution is a
sum of a particular solution and the homogeneous solution ($f\equiv0$)
with perturbed boundary conditions.} Therefore, we need to make sure
that there exists no unstable homogeneous solution. 

After Fourier transform in the $x$ variable, \eqref{eq:model-extended-Laplace-Beltrami}
and \eqref{eq:perturbed-Neumann} become 
\[
\mu\hat{v}_{yy}(\omega,y)=(\omega^{2}+c)\hat{v}(\omega,y),\,\,\,\omega\in2\pi\mathbb{Z},\,\,\,y\in(-\epsilon,\epsilon),
\]
\[
\frac{\partial\hat{v}}{\partial y}=\begin{cases}
\alpha(i\omega)^{k}\hat{v}, & y=\epsilon,\\
0, & y=-\epsilon.
\end{cases}
\]
The general solution of this two-point boundary value problem takes
the form
\[
\hat{v}(\omega,y)=c_{1}e^{\kappa y}+c_{2}e^{-\kappa y},\,\,\,\kappa=\sqrt{\frac{\omega^{2}+c}{\mu}}>0.
\]
 We first observe that if there is a non-trivial solution, then the
magnitude of $\kappa\epsilon=\sqrt{\frac{\omega^{2}+c}{\mu}}\epsilon$
determines an upper bound of the solution's $y$-derivative, which
translates to how far the solution is from being constant-along-normal.
We only need to make sure that the non-trivial solution exists when
$|\omega|$ is asymptotically smaller than $\epsilon^{-1}$. 

From the boundary conditions at $y=\pm\epsilon$, we obtain
\begin{align*}
c_{1}\kappa e^{-\kappa\epsilon}-c_{2}\kappa e^{\kappa\epsilon} & =0,\\
c_{1}\kappa e^{\kappa\epsilon}-c_{2}\kappa e^{-\kappa\epsilon} & =\alpha(i\omega)^{k}(c_{1}e^{\kappa\epsilon}+c_{2}e^{-\kappa\epsilon}),
\end{align*}
After simplification, we obtain 
\begin{equation}
\kappa\,\tanh(2\kappa\epsilon)=\alpha(i\omega)^{k},\label{eq:general_form_of_eigen_value_problem}
\end{equation}
when $c_{1},c_{2}\neq0$. We see that the existence of solutions to
the eigenvalue problem \eqref{eq:general_form_of_eigen_value_problem}
can be classified according to $k$ being an even or odd positive
integer.

\paragraph{$k$ is odd. }

We see that in this case there exists no solution to the eigenvalue
problem \eqref{eq:general_form_of_eigen_value_problem}. Therefore
the only solution is the trivial solution $\hat{v}=0.$ 

\paragraph{$k$ is even. }

More explicitly, the left hand side of equation \eqref{eq:general_form_of_eigen_value_problem}
is 
\[
L(\omega):=\sqrt{\frac{\omega^{2}+c}{\mu}}\tanh\left(2\sqrt{\frac{\omega^{2}+c}{\mu}}\epsilon\right)
\]
 and the right hand side
\[
R(\omega):=\tilde{\alpha}\omega^{k},
\]
where $\tilde{\alpha}=i^{k}\alpha$. If $\tilde{\alpha}\leq0$, $L(\omega)=R(\omega)$
has no solution. If $\tilde{\alpha}>0$, $L(\omega)-R(\omega)$ changes
signs three times as $\omega\rightarrow\pm\infty$, so there are two
roots. Furthermore, for $|\omega|\gg1$, to leading order, $L(\omega)\simeq\frac{|\omega|}{\sqrt{\mu}}$
and $R(\omega)\simeq\tilde{a}\omega^{k}$ \textendash{} there is no
solution of order $\epsilon^{-1}.$ 

Thus, we conclude that the proposed extension has a solution which
is stable with respect to the perturbation in the boundary condition
\eqref{eq:perturbed-Neumann}, imposed on $\partial T_{\epsilon}.$

\subsubsection{Stability of gradient flows}

The gradient descent equation is 
\[
v_{t}=v_{xx}+\mu v_{yy}-cv,\,\,\forall(x,y)\in[0,1)\times(-\epsilon,\epsilon),\quad v_{y}(x,-\epsilon,t)=0,\,v_{y}(x,\epsilon,t)=\alpha\frac{\partial^{k}v}{\partial x^{k}},
\]
and $v$ is periodic in $x$. The heat equation with the perturbed
Neumann boundary condition is stable in the energy norm for most cases.
Let $||\cdot||$ be the standard $L^{2}([0,1)\times(-\epsilon,\epsilon))$
norm. From the energy estimate
\begin{align*}
\frac{1}{2}\frac{d}{dt}||v||^{2} & =-\left(||v_{x}||^{2}+\mu||v_{y}||^{2}+c||v||^{2}\right)+\int_{0}^{1}\mu(vv_{y}(x,\epsilon,t)-vv_{y}(x,-\epsilon,t)dydx\\
 & =-\left(||v_{x}||^{2}+\mu||v_{y}||^{2}+c||v||^{2}\right)+\int_{0}^{1}\alpha\mu v\frac{\partial^{k}v}{\partial x^{k}}(x,\epsilon,t)dx,
\end{align*}
one can deduce the following conclusion: 

\paragraph*{$k$ is odd}

One can easily show that $\int_{0}^{1}v\frac{\partial^{k}v}{\partial x^{k}}dx=0$
by applying integration by parts several times and $v$ is periodic
in $x$. Therefore $\frac{d}{dt}||v||^{2}=-2\left(||v_{x}||^{2}+\mu||v_{y}||^{2}+c||v||^{2}\right)$
and the solution is stable.

\paragraph{$k$ is even}

Let $k=2m$. By applying integration by parts several times and $v$
is periodic in $x$, we have $\int_{0}^{1}v\frac{\partial^{k}v}{\partial x^{k}}\,dx=(-1)^{m}\int_{0}^{1}(\frac{\partial^{m}v}{\partial x^{m}})^{2}\,dx$.
If $\tilde{\alpha}=i^{k}\alpha=(-1)^{m}\alpha\leq0$, then $\frac{d}{dt}||v||^{2}=-2\left(||v_{x}||^{2}+\mu||v_{y}||^{2}+c||v||^{2}\right)+\tilde{\alpha}\|\frac{\partial^{m}v}{\partial x^{m}}\|^{2}$
and the solution is stable. If $\tilde{\alpha}=i^{k}\alpha=(-1)^{m}\alpha>0$,
we can not conclude that the solution is stable. 

In fact, there may exist unstable solutions. Consider $c=\mu=1,k=2$
and $\alpha=-1$. Again, we look for a solution of the form
\[
v(x,y,t)=e^{st+\omega xi}(e^{\kappa y}+e^{-k(y+2\epsilon)}).
\]
Then $v$ is a solution if $s=\kappa^{2}-\omega^{2}-1$ and $\kappa\tanh(2\kappa\epsilon)=\omega^{2}.$
It is not hard to show that the equation $\kappa\tanh(2\kappa\epsilon)=\omega^{2}$
has real solutions for $\omega=2n\pi$ and hence the equation has
unstable solutions if $s>0$. In particular, if we choose $\epsilon=0.01$
and $\omega=2\pi$, then $\kappa\simeq51.1815$ is a root of $\kappa\tanh(0.02\kappa)=(2\pi)^{2}$
and in this case $s=\kappa^{2}-(2\pi)^{2}-1\simeq2579>0$. The solution
$v$ grows exponentially and the equation admits unstable modes. 

As we discuss in Section \ref{subsec:proposed-BC}, in our simulations
the perturbed boundary is of the form 

\[
\frac{\partial v}{\partial\mathbf{n}}\simeq ch^{3}\frac{\partial^{4}v}{\partial\mathbf{t}^{4}},
\]
with some positive number $c$. Therefore the proposed method with
cubic interpolating polynomials in a dimension-by-dimension fashion
is stable for gradient descent equations. 

\subsection{Constraints\label{subsec:Constraints}}

In the model problem \eqref{eq:model-in-the-strip}, if $c=0,$ the
solutions are unique up to an overall additive constant. The addition
of $\mu v_{yy}$ together with \eqref{eq:model-problem:v_y=00003D0}
will not make the solution unique. In this case, one must extend the
additional constraint on $u$ consistently so ensure the uniqueness
of solution. For example, if one minimizes $I_{\Gamma}[u]$ subject
to the constraint $\int_{\Gamma}u^{2}ds=C_{0}$, then one should impose
$\int_{T_{\epsilon}}v^{2}(z)K_{\epsilon}(y)dz=C_{0}$ for the minimization
of $\tilde{I}_{T_{\epsilon}}[v].$

We see that under the proposed framework, integral constraints of
the type $\int_{\Gamma}G(u)ds=C_{0}$ can easily be compatibly extended
to $\int_{T_{\epsilon}}G(v)\,K_{\epsilon}(d_{\Gamma})Jdz=C_{0}.$
In general, using Lagrange multiplier, this type of constrained minimization
problem is reduced to nonlinear Eigenvalue problems. 

Let us demonstrate the solution of a simple model problem via gradient
descent: 
\[
\min_{u}\int_{\Gamma}|\nabla u|^{2}\,dS\,\,\,\text{subject to }\int_{\Gamma}u^{2}dS=C_{0}.
\]
 One way to solve this problem is to introduce a Lagrange multiplier
and solve to steady state the following equation:
\begin{equation}
u_{t}=\Delta_{\Gamma}u+\lambda(t)u,\label{eq:surface_heat_with_lagrange_multiplier}
\end{equation}
where the multiplier $\lambda$ is chosen to be
\begin{equation}
\lambda(t)=\frac{-\int_{\Gamma}u\Delta_{\Gamma}udS}{\int_{\Gamma}u^{2}dS}=\frac{\int_{\Gamma}|\nabla_{\Gamma}u|^{2}dS}{\int_{\Gamma}u^{2}dS}.\label{eq:Lagrange-multiplier}
\end{equation}
Then we can check 
\[
\frac{d}{dt}\int_{\Gamma}u^{2}dS=2\int_{\Gamma}uu_{t}dS=2\int_{\Gamma}u\Delta_{\Gamma}u+\lambda u^{2}dS=0
\]
It is rather straight forward to extend \eqref{eq:surface_heat_with_lagrange_multiplier}
and \eqref{eq:Lagrange-multiplier} following the proposed framework.
We present a computational result computing the solution of the $p-$Laplacian
on torus with an integral constraint in Section \ref{subsec:Minimization-of--Laplacian}.

\subsection{Gradient flow\label{subsec:Gradient-flow} }

The gradient flow of $\tilde{I}_{T_{\epsilon}}[v(\cdot,t)]$ should
be consistent with the $L^{2}$ gradient descent of $I_{\Gamma}[u(\cdot,t)]$.
Again, letting $\bar{u}(z,t)=u(P_{\Gamma}z,t),$ we see that 

\begin{align*}
\begin{split}\frac{d}{dt}I_{\Gamma}[u(\cdot,t)]=\frac{d}{dt}\tilde{I}_{T_{\epsilon}}[\bar{u}(\cdot,t)] & =\frac{1}{2\epsilon}\int_{T_{\epsilon}}\left((-J^{-1}\nabla\cdot(JAL_{{\bf p}})+L_{q})\,\bar{u}_{t}\right)\,Jdz\end{split}
\\
 & =\frac{1}{2\epsilon}\int_{T_{\epsilon}}(-\nabla\cdot(JAL_{{\bf {\bf p}}})+JL_{q})\,\bar{u}_{t}dz.
\end{align*}
This means that in order to have the same gradient flow for the integral
defined by each of the parallel surface, one should consider the $J$-weighted
$L^{2}$ norm for choosing $v_{t}$. The subtlety here is that the
consistent gradient flow of $\tilde{I}_{T_{\epsilon}}$ should be
\[
v_{t}=J^{-1}(\nabla\cdot(JAL_{{\bf p}}(P_{\Gamma}z,v,A\nabla v))+L_{q}(P_{\Gamma}z,v,A\nabla v)),
\]
 instead of
\[
v_{t}=\nabla\cdot(JAL_{{\bf p}}(P_{\Gamma}z,v,A\nabla v))+JL_{q}(P_{\Gamma}z,v,A\nabla v),
\]
even though both equations yield the same steady state (for strictly
convex $\tilde{I}_{T_{\epsilon}}$). 

Finally,\textcolor{blue}{{} }as the time independent cases we have shown
in Theorem \ref{thm:Equivalent}, it can be shown that $v(\cdot,0)\equiv u(\cdot,0)\implies v(\cdot,t)\equiv u(\cdot,t)$;
i.e. that Assumption~\eqref{assu:dv/dn=00003D0} holds for $t>0$
if the initial condition satisfies it. In Section \ref{subsec:Comparison-of-decaying}
we give an example of the rate of decay of $I_{\Gamma}$ and $I_{T_{\epsilon}}$
computed correctly and incorrectly.

\subsection{Tests cases }

In the following, we present a few numerical experiments to study
the proposed formulation and the issues discussed above. For simplicity,
the curvatures and Jacobians used in the following numerical simulations
are computed analytically instead of approximating $DP_{\Gamma}$
by finite differences and compute its singular values.

\subsubsection{Sensitivity of boundary closure to the depth of projection\label{subsec:Sensitivity-of-boundary}}

Consider the Laplace-Beltrami problem $-\Delta_{\Gamma}u=f$ on a
unit sphere $\Gamma$. The extended problem becomes

\[
-r^{2}\Delta v=\overline{f}\quad on\quad T_{\epsilon},\quad\text{and}\quad\frac{\partial v}{\partial n}=0\quad\text{on}\quad\partial T_{\epsilon},
\]
where $r=\sqrt{x^{2}+y^{2}+z^{2}}$ if we choose $\mu=r.$ In our
numerical simulation, we pick $\overline{f}=2\frac{x}{r}$ and the
solution is $u=\frac{x}{r}+c$ for any constant $c$. We use $h=\Delta x=\Delta y=\Delta z=\frac{1}{20}$
and the bandwidth $\epsilon=10h$ in our numerical simulations in
which the boundary closure is carried our at four different ``depths''
$\alpha$. This means that the ghost node $(ih,jh,kh)$ is projected
along the normal, $\alpha$ distance into $T_{\epsilon}$, i.e. 
\[
(ih,jh,kh)\mapsto(ih,jh,kh)-\alpha\mathbf{n}_{i,j,k},
\]
with $\alpha=d(x)-(10-2\sqrt{3})h,\,d(x)-3h,\,d(x),\,d(x)+3h$. We
discretize the PDE by the second order central difference scheme and
use cubic polynomial to interpolate ghost nodes. The resulted linear
system is singular since the solution is unique up to an additive
constant. Therefore we choose the exact solution to be $u=\frac{x}{r}$
and force the first index node in our numerical solution has the same
value as the exact solution. After deleting the first column and first
row in the linear system resulting form the discretization, it is
invertible. The condition number of the reduced linear system and
the error are listed in Table \ref{tab:interpolation_position}. We
can see if $\alpha$ become larger, then the condition number and
error are worse. The structure of the matrices are show in Figure
\ref{fig:structure of matrices}. Notice that the smaller $\alpha$
gives the tighter band structure of the matrix.

Next we study the sensitivity of condition number of the reduced linear
system to the distance between $\partial T_{\epsilon}$ and the interior
grid node closest to it. In certain finite element formulation \cite{olshanskii2010finite},
the smallness of this distance could render the resulting linear system
ill-conditioned. We use $h=\Delta x=\Delta y=\Delta z=\frac{1}{10}$
and $\frac{1}{20}$ with the bandwidth $\epsilon=3h$ in our numerical
simulations.  We perturb the grids $h\mathbb{Z}^{3}\cap T_{\epsilon}$
by $(h\mathbb{Z}^{3}+{\bf r})\cap T_{\epsilon}$ with a random vector
${\bf r}=(r_{x},r_{y},r_{z})h$, where $r,r_{y},r_{z}$ are randomly
generated from uniform distribution on {[}-0.5, 0.5). We repeat 100
times. The minimum distance between interior nodes and $\partial T_{\epsilon}$
versus the condition number of the reduced linear system is shown
in Figure \ref{fig:Distribution of CN}. It shows our method is stable
even when some interior nodes are very close to the boundary of the
computational domain $T_{\epsilon}$.

\begin{table}
\caption{Comparison of different choices of positions to interpolate boundary
points. We can see the closer to the boundary points the better condition
number of matrices and error.\label{tab:interpolation_position}}

\centering{}\medskip{}
\begin{tabular}{|c|c|c|c|c|}
\hline 
$\alpha$ & $d(x)-(10-2\sqrt{3})h$ & $d(x)-3h$ & $d(x)$ & $d(x)+3h$\tabularnewline
\hline 
\hline 
$\kappa(M)$ & 1.4387 e+7  & 1.7978 e+7 & 2.2930 e+7  & 3.1236 e+7\tabularnewline
\hline 
Error & 4.0124 e-4 & 3.9211 e-4 & 4.0148 e-4 & 4.3044 e-4\tabularnewline
\hline 
\end{tabular}
\end{table}

\begin{figure}
\begin{centering}
\includegraphics[scale=0.6]{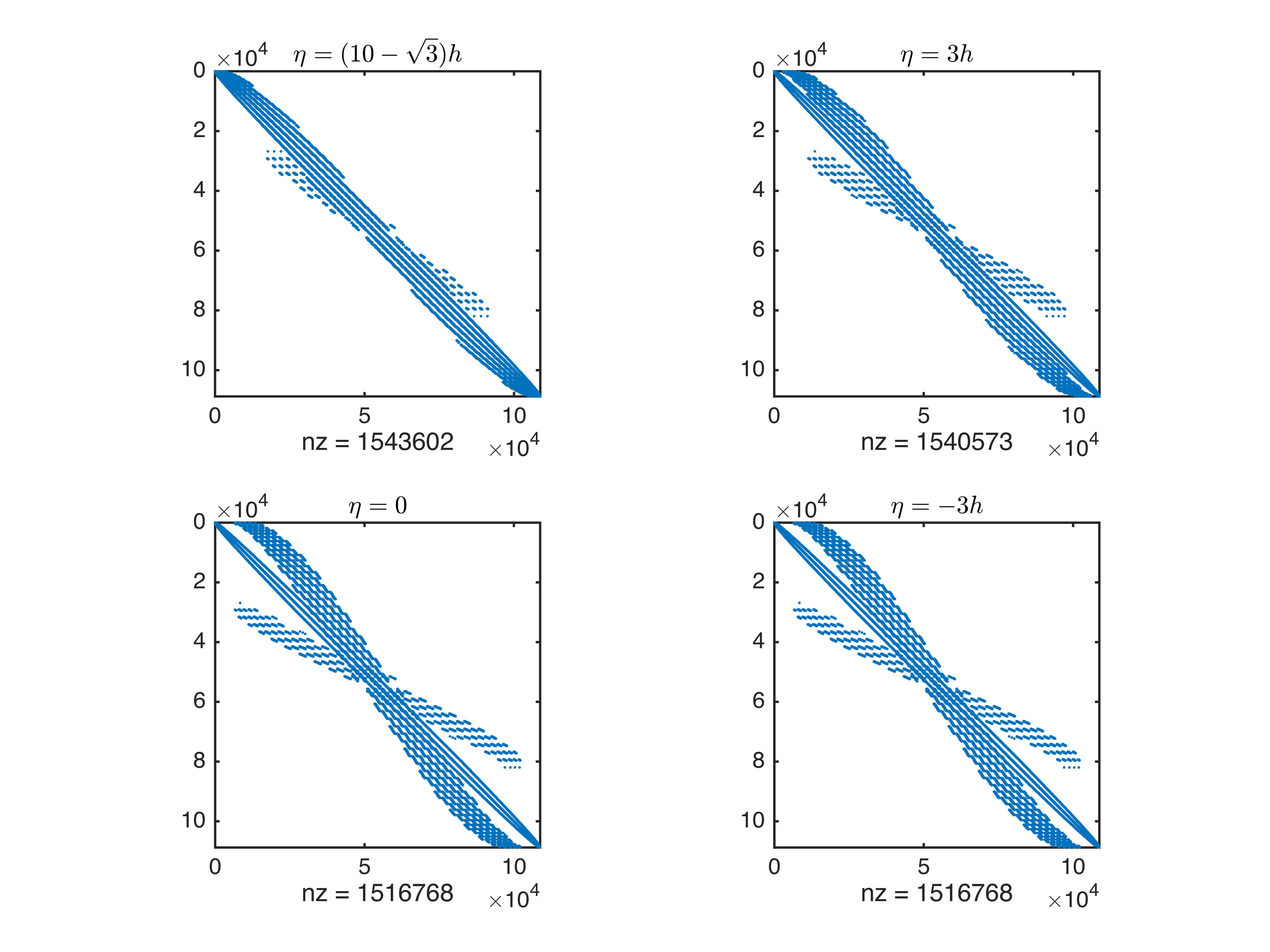}
\par\end{centering}
\caption{The structures of the matrices obtained from different $\eta=d(x)-\alpha$.
The variable nz refers to the size of the matrices.\label{fig:structure of matrices}}
\end{figure}

\begin{figure}
\begin{centering}
\includegraphics{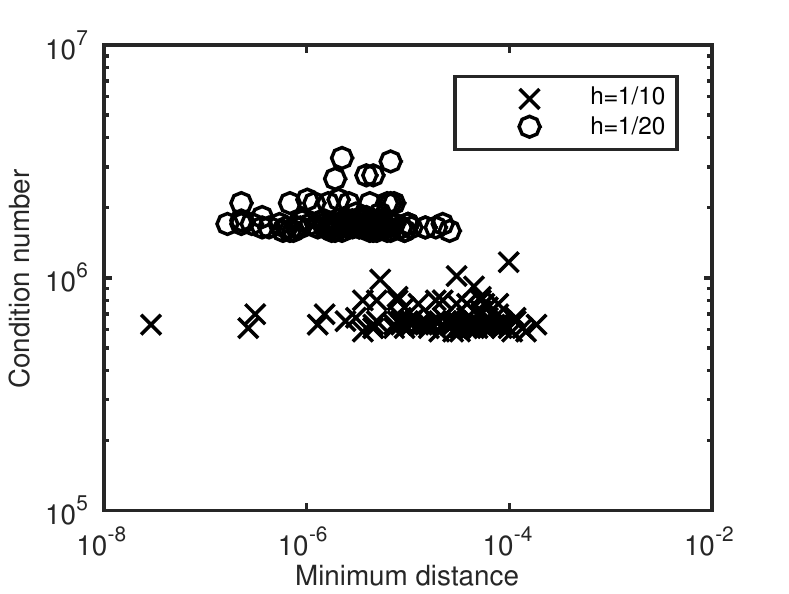}
\par\end{centering}
\caption{Distribution of minimum distances between inner nodes and the boundary
of $T_{\epsilon}$ versus the condition numbers of reduced linear
system among 100 times random perturbation on grid positions. \label{fig:Distribution of CN}}
\end{figure}

\subsubsection{Eigenfunctions of the Laplace-Beltrami operator }

\paragraph{On a circle in three dimensions}

We consider the Laplace-Beltrami eigenvalue problem on a unit circle
in $\mathbb{R}^{3}$ defined by 
\begin{equation}
x(\theta)=\left[\begin{array}{l}
x_{1}(\theta)\\
x_{2}(\theta)\\
x_{3}(\theta)
\end{array}\right]=\left[\begin{array}{ccc}
-\frac{7}{\sqrt{102}} & \frac{1}{\sqrt{2}} & \frac{1}{\sqrt{51}}\\
-\frac{7}{\sqrt{102}} & -\frac{1}{\sqrt{2}} & \frac{1}{\sqrt{51}}\\
\frac{2}{\sqrt{102}} & 0 & \frac{7}{\sqrt{51}}
\end{array}\right]\left[\begin{array}{c}
\cos\theta\\
\sin\theta\\
0
\end{array}\right],\,\,\,0\le\theta<2\pi.\label{eq:circle3d_para}
\end{equation}
This circle is not symmetric in any way with respect to the underlying
Cartesian grid, and we discretize this problem as if we are dealing
with more general smooth curves without using the knowledge of the
parametrization of the circle. Recall that the Laplace-Beltrami eigenvalue
problem is to solve

\begin{equation}
\Delta_{\Gamma}u_{n}=\lambda_{n}u_{n}.
\end{equation}
For simplicity, the kernel function is chosen to be constant. We choose
two free convexification constants $\mu_{1}=\mu_{2}=\sigma^{-1}=1-d_{\Gamma}(p)\kappa(p)$,
where $\kappa$ is the curvature of the parallel curve at point $p$.
Hence $A=\sigma^{-1}I_{3}$ is a scalar tensor and the extended PDE
on $T_{\epsilon}$ is
\begin{equation}
\sigma^{-1}\text{\ensuremath{\nabla}}\cdot(\sigma^{-1}\nabla v_{n})=\lambda v_{n}\label{eq:circle3d}
\end{equation}
In this case, $\sigma^{-1}=\sqrt{u_{1}^{2}+u_{2}^{2}},$ where $[u_{1},u_{2},u_{3}]^{T}=Rx$
and $R$ is the orthogonal matrix in \eqref{eq:circle3d_para}. The
first eigenvalue is 0 and simple and the corresponding eigenfunction
is constant. The rest of the eigenvalues are $n^{2}$ for $n\in\mathbb{N}$,
with multiplicity 2. The corresponding eigenfunctions are $\cos(n\theta+\phi)$
for any arbitrary phase shift $\phi$. 

The numerical simulations are carried out on the grid nodes in $T_{\epsilon}\cap h\mathbb{Z}^{3}$
with $\epsilon=4h$, and $h$ is chosen to be $\frac{1}{10},\frac{1}{20},\frac{1}{40},\frac{1}{80},\frac{1}{160},\frac{1}{320}$.
We discretize \eqref{eq:circle3d} by second order central difference
and computed the eigenvalues of the discretized system as approximations
of exact eigenvalues. The errors of eigenvalues are listed in Table
\ref{tab:LB-circle-3D}. We can see the convergence rate is close
to second-order. In Figure \ref{fig:eigenfun_circle_25_theta}, we
plot the eigenfunction corresponding to $\lambda=25$ on the equal
distance surface $\Gamma_{2h}$ on the left hand side. The eigenfunction
is close to constant on the normal cross section surface. If we projected
the eigenfunction back to the unit circle, we can see it is close
to $\cos5\theta$ after suitable phase shift in right subfigure of
Figure \ref{fig:eigenfun_circle_25_theta}. The error of the eigenfunction
in this case is of magnitude $10^{-5}$.

\begin{figure}
\begin{centering}
\includegraphics{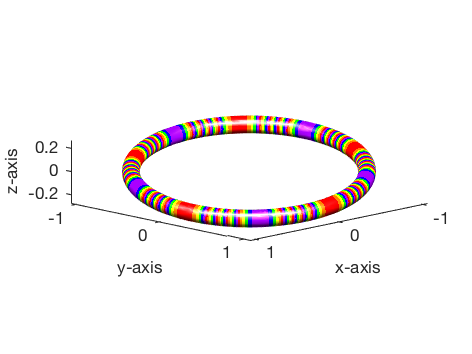}\includegraphics{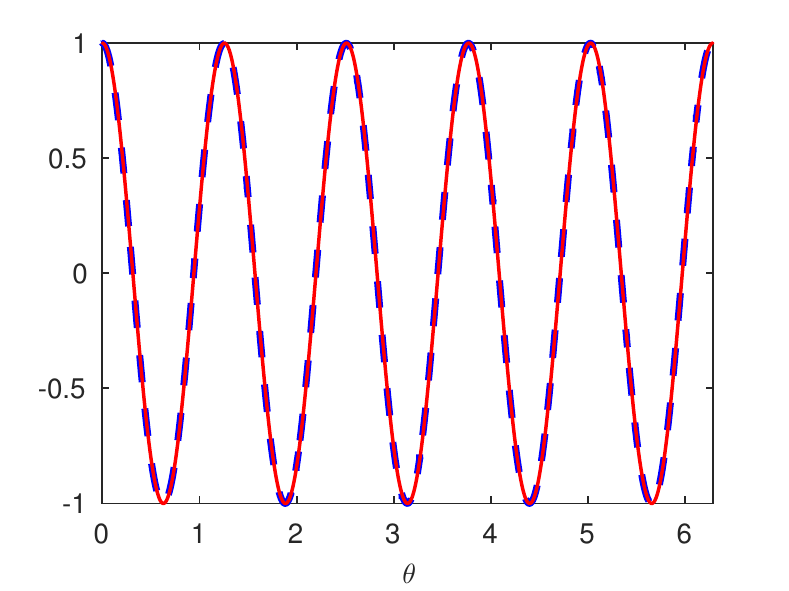}
\par\end{centering}
\caption{The eigenfunction of unit circle corresponding to $\lambda=25$. (Left)
The eigenfunction on the torus surface that is the equal distant surface
$\Gamma_{2h}$ with $h=1/20$. (Right) The eigenfunction with respect
to the $\theta$ parameter on the curve. The blue dashed curve is
numerical solution and the red solid curve is $\cos5\theta$. \label{fig:eigenfun_circle_25_theta}}
\end{figure}

\begin{table}
\caption{\label{tab:LB-circle-3D}The errors of the Laplace-Beltrami eigenvalue
problem for a unit circle in $\mathbb{R}^{3}$.}

\centering{}\medskip{}
\begin{tabular}{|c|c|c|c|c|}
\hline 
$h^{-1}$ & Error for $\lambda$ = 1 & Error for $\lambda$ = 4 & Error for $\lambda$ = 25 & Error for $\lambda$ = 36\tabularnewline
\hline 
\hline 
10 & 6.4982e-4 & 5.4573e-2 & 3.9316e-2 & 1.7026e-1\tabularnewline
\hline 
20 & 6.2345e-5 & 1.5447e-3 & 9.2460e-3 & 3.4550e-2\tabularnewline
\hline 
40 & 5.7136e-5 & 6.5797e-4 & 2.6530e-3 & 3.7635e-3\tabularnewline
\hline 
80 & 3.1412e-6 & 1.8924e-4 & 2.8500e-3 & 6.9070e-3\tabularnewline
\hline 
160 & 7.0887e-7 & 3.4274e-4 & 1.4435e-4 & 6.4697e-4\tabularnewline
\hline 
320 & 7.0887e-7 & 1.5974e-5 & 5.4665e-5 & 4.2973e-4\tabularnewline
\hline 
Order & 1.8814   & 1.9144 & 1.8672  & 1.6997\tabularnewline
\hline 
\end{tabular}
\end{table}

\paragraph{On a torus in three dimensions}

Let $\Gamma$ be the torus $(\sqrt{x^{2}+y^{2}}-0.7)^{2}+z^{2}=0.3^{2}$
in $\mathbb{R}^{3}$. Consider the Laplace-Beltrami eigenvalues $\lambda_{n}$
and eigenfunctions $\phi_{n}$ of $\Gamma$:

\[
-\Delta_{\Gamma}\phi_{n}=\lambda_{n}\phi_{n}.
\]
We compute the eigenvalues and eigenfunctions by solving the extended
eigenvalue problems:

\[
-J^{-1}\nabla\cdot(JA^{2}\nabla\psi_{n})=\lambda_{n}\psi_{n}\,\,\text{in}\,T_{\epsilon},\,\,\frac{\partial\psi_{n}}{\partial n}=0\,\,\text{on}\,\partial T_{\epsilon}.
\]
In this numerical simulation, we use grid nodes in $T_{\epsilon}\cap h\mathbb{Z}^{3}$
with $\epsilon=4h$ and $h$ is chosen to be $\frac{1}{20}.$ We compute
the discrete approximation of the differential operator of the extended
equation. Then solve the eigenvalues and eigenvectors of the induced
linear system. Figure \ref{fig: e-vec of Laplace Beltrami on a torus}
shows a computational result of an eigenfunction corresponding to
the 6th eigenvalue $\lambda_{6}\simeq6.51$ of the Laplace-Beltrami
operator on the torus $\Gamma$.

\begin{figure}
\begin{centering}
\includegraphics[height=2in]{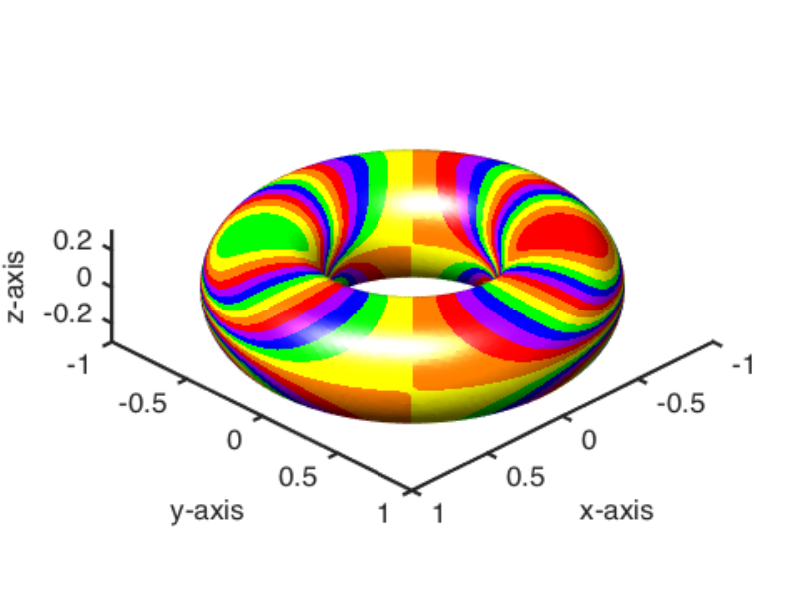}\includegraphics[height=1.75in]{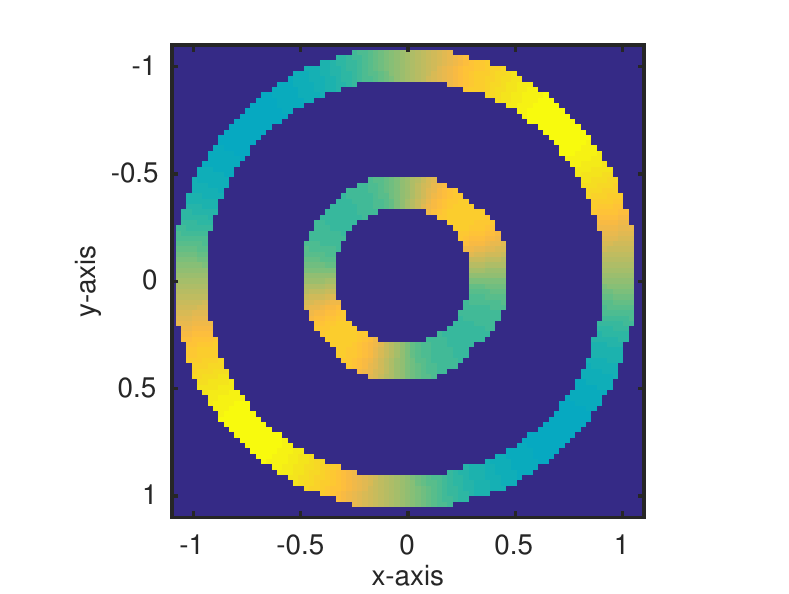}
\par\end{centering}
\caption{An eigenfunction of the Laplace-Beltrami operator on a torus. The
results are computed by the proposed algorithm. (Left) The colormap
used here reveals the structure of the eigenfunctions' level lines.
(Right) The $z=0$ section of the solution, indicating that the computed
solution is indeed constant-along-normal. \label{fig: e-vec of Laplace Beltrami on a torus}}
\end{figure}

\subsubsection{Energy decay rates \label{subsec:Comparison-of-decaying}}

Let $\Gamma$ be the ellipse with major axis equal to 4 and minor
axis equal to 2 in $\mathbb{R}^{2}$. The angle between the major
axis and positive $x-$axis is set to be $\frac{\pi}{5}$. See Figure
\ref{fig:domain_ellipse} for illustration for the computational domain.
The ellipse can be parametrized by 
\[
\left[\begin{array}{l}
x(\theta)\\
y(\theta)
\end{array}\right]=\left[\begin{array}{rr}
\cos\frac{\pi}{5} & -\sin{\frac{\pi}{5}}\\
\sin\frac{\pi}{5} & \cos{\frac{\pi}{5}}
\end{array}\right]\left[\begin{array}{r}
2\cos\theta\\
\sin\theta
\end{array}\right].
\]
Notice that $\theta$ is not arc-length parameter and the arc-length
function $s(\theta)=\int_{0}^{\theta}(3\sin^{2}\beta+1)\,d\beta$
is computed numerically. The energy function $I_{\Gamma}$ is given
by $I_{\Gamma}(u)=\frac{1}{2}\int_{\Gamma}|\nabla_{\Gamma}u|^{2}\,dS=\frac{1}{2}\int_{\Gamma}u_{s}^{2}\,dS$.
The gradient decent flow $u(t,s)$ of $I_{\Gamma}$ satisfies the
heat equation 
\begin{equation}
u_{t}=u_{ss}\label{eq:heat_gamma}
\end{equation}
on $\Gamma$. If we choose free variable $\mu=\sigma^{-1}$ and constant
kernel function, the extended energy function $J_{T_{\epsilon}}$
of $I_{\Gamma}$ is
\begin{equation}
J_{T_{\epsilon}}(v)=\frac{1}{2}\int_{T_{\epsilon}}\sigma^{-1}|\nabla v|^{2}\:dz.
\end{equation}
As discussed in Section \ref{subsec:Gradient-flow}, instead of solving
\begin{equation}
v_{t}=\nabla\cdot(\sigma^{-1}\nabla v),\label{eq:heat_T_wrong}
\end{equation}
the gradient flow of $J_{T_{\epsilon}}$ should be 
\begin{equation}
v_{t}=\sigma^{-1}\nabla\cdot(\sigma^{-1}\nabla v).\label{eq:heat_T_right}
\end{equation}
In the numerical simulation, the initial data is set to be $u(s,0)=\sin(ks)$,
where $k=\frac{2\pi}{L}$ and $L\simeq9.6884$ is the total length
of the ellipse $\Gamma$. The exact solution of \eqref{eq:heat_gamma}
is $u(s,t)=e^{-k^{2}t}\sin(ks)$. We apply forward Euler in time and
discretize the equation by using $\Delta x=\Delta y=h=\frac{1}{100}$
and $\Delta t=\frac{h^{2}}{10}$ with $\epsilon=3h$. We compare the
$L_{\infty}$-error $\|v(s,\eta,t)-e^{-k^{2}t}\sin(ks))\|_{\infty,T_{\epsilon}}$
on $T_{\epsilon}$ and the energy decaying rate with the solutions
obtained from \eqref{eq:heat_T_wrong}, \eqref{eq:heat_T_right} numerically
and from the closest point method \cite{ruuth_merriman08}.   {We
use the standard second order central difference scheme for $\nabla\cdot(\sigma^{-1}\nabla v)$
and bi-cubic interpolation in the boundary closure. The energy errors
are integrated numerically using \eqref{eq:real-def-of-J_Teps} with
$\bar{L}=\frac{1}{2}|\nabla v|^{2}$and $K_{\epsilon}(d)=\frac{1}{2\epsilon}(1+\cos(\frac{d}{\epsilon}))\chi_{[-1,1]}(d)$
where $\chi_{[-1,1]}$ is the indicator function of $[-1,1]$. \eqref{eq:real-def-of-J_Teps}
is discretized be central differencing for $\nabla v$ and trapezoidal
rule for the integral. For close point method, we use second order
central difference scheme to compute Laplacian and bi-cubic interpolation
for close point extension procedure. In computing the energy error
in the solution of the closest point method, $|\nabla v|^{2}$ is
replaced by the values on $\Gamma$, i.e.} $|\nabla v(P_{\Gamma}x)|^{2}$.
The results are shown in Figure \ref{fig:energy decay}. We can see
the proposed method with correct decaying rate has smallest error
in both $L_{\infty}$-norm and energy norm.

Next, we apply Crank-Nicolson time discretization for this example.
The resulting matrices are inverted by the built-in GMRES scheme in
Matlab to solve the linear system with tolerance $10^{-13}$. In Table~\ref{tab:Implicit},
we report certain aspects of the computational results obtained with
$h=\frac{1}{50},\frac{1}{100},\frac{1}{200},\frac{1}{400},\frac{1}{800,}\frac{1}{1600}$
, $\Delta t=\frac{h}{10}$ , $\epsilon=3h$; these include the condition
numbers of the matrices $I-\frac{\Delta t}{2}Q_{h}$, where $Q_{h}$
corresponds to the centered differencing discretization of the spatial
derivatives, the average of numbers of iterations and the $L_{\infty}$-errors
$\|v(s,\eta,t)-e^{-k^{2}t}\sin(ks))\|_{\infty,T_{\epsilon}}$ on $T_{\epsilon}$
for $0\le t\le2.$ The condition number scales like $O(\frac{1}{h})$
and the convergence rate is second-order as we expect. 

\begin{figure}
\begin{centering}
\includegraphics{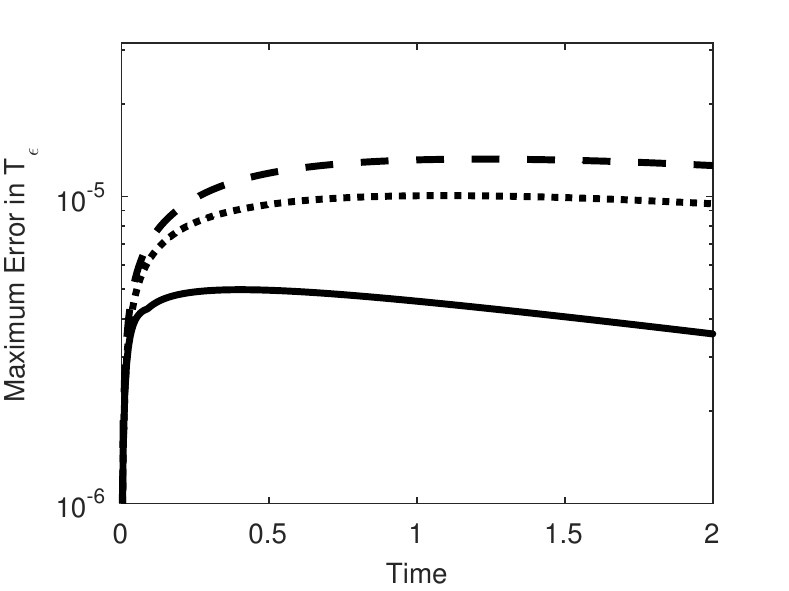}\includegraphics{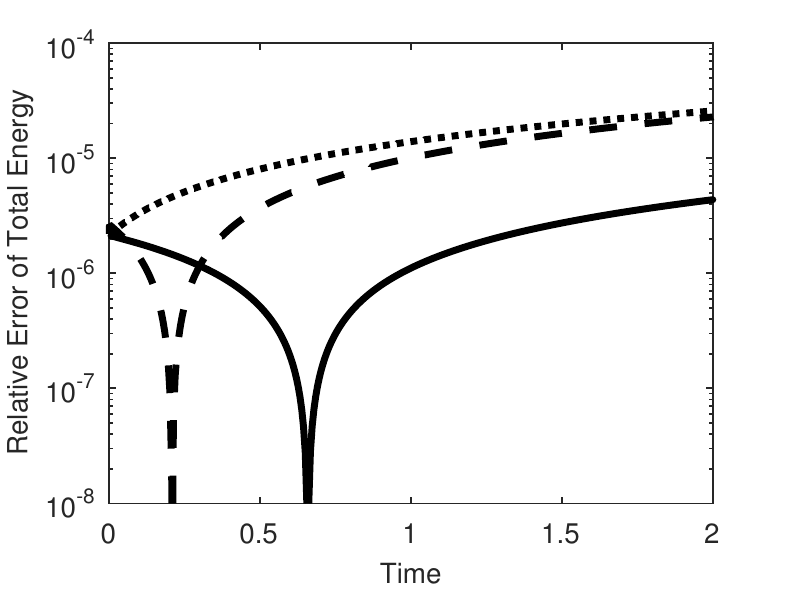}
\par\end{centering}
\caption{(Left) The $L_{\infty}$ -error on the narrowband $T_{\epsilon}$.
(Right) The relative error of energy on the narrow band $T_{\epsilon}$.
The solid curves are obtained by proposed method solving \eqref{eq:heat_T_right}
and the dashed curves are with \eqref{eq:heat_T_wrong}. The dotted
curves are obtained by the closest point method. \label{fig:energy decay}}
\end{figure}

\begin{table}
\caption{The performance of the proposed method with Crank-Nicolson method
in time.\label{tab:Implicit}}
\medskip{}

\centering{}%
\begin{tabular}{|c|c|c|c|c|}
\hline 
$h^{-1}$ & Condition number  & \# of iterations & $L_{\infty}$-error & Order\tabularnewline
\hline 
\hline 
50 & 23.94 & 8.2 & 1.4459e-5 & \tabularnewline
\hline 
100 & 45.00 & 9.7 & 3.5625e-6 & 2.0210\tabularnewline
\hline 
200 & 88.82 & 9.7 & 9.2893e-7 & 1.9393\tabularnewline
\hline 
400 & 176.70 & 12.6 & 2.2653e-7 & 2.0358\tabularnewline
\hline 
800 & 353.24 & 17.6 & 5.7933e-8 & 1.9673\tabularnewline
\hline 
1600 & 708.40 & 25.5 & 1.4441e-8 & 2.0042\tabularnewline
\hline 
\end{tabular}
\end{table}

\subsubsection{Gradient flow of Allen-Cahn equation}

Let $\Gamma$ be the ellipse given in Section \ref{subsec:Comparison-of-decaying}.
We consider the Modica-Mortola energy on $\Gamma$:

\begin{equation}
I_{\Gamma}(u)=\int_{\Gamma}\frac{\delta}{2}\|\nabla_{\Gamma}u\|^{2}+\frac{1}{4\delta}(1-u^{2})^{2}\,dS,\,\,\,0<\delta\ll1,
\end{equation}
and its gradient descent by the Allen-Cahn equation on $\Gamma$:

\begin{equation}
u_{t}=u_{ss}+\frac{1}{\delta^{2}}u(1-u^{2}).\label{eq:Allen-Cahn}
\end{equation}
The extended equation of \eqref{eq:Allen-Cahn} is simply 

\[
v_{t}=\sigma^{-1}\nabla\cdot(\sigma^{-1}\nabla v)+\frac{1}{\delta^{2}}v(1-v^{2}),\quad\text{on}\,\,T_{\epsilon},
\]
with zero Neumann boundary condition. Since the Modica-Mortola energy
is not strictly convex, the equation has non-unique minimizers and
the minimizers depend on initial conditions. In this numerical experiment,
we pick $\delta=0.03$. We use forward Euler in time and choose $h=\frac{1}{200}$,
$\epsilon=0.3$ with $\Delta t=\frac{h^{2}}{10}$. The initial condition
is set to be 

\[
v(x,0)=\sum_{j=11}^{13}\sin\left(\frac{2j\pi s^{2}}{L^{2}}\right),
\]
where $L\simeq9.6884$ is the total length of the ellipse $\Gamma$
and $s$ is the arc-length parameter on $\Gamma$. Since there is
no analytical form for the exact solution, we use very fine grid (2000
equidistant points on $\Gamma$) and apply central difference scheme
with forward Euler in time to discretize \eqref{eq:Allen-Cahn} directly
as the reference solution. We compare our proposed method with a modified
version of the closest point method (CPM).  {{} This modified
CPM is not recommended in practice. We use this example to illustrate
that CPMs have an internal time scale that is defined by the frequency
of the ``reinitialization'' steps, and successful application of
a CPM can depend on such internal time scale and its relation to the
intrinsic time scale of the problem. }

In the modified CPM, we solve

\[
v_{t}=\Delta v(cp(x))+\frac{1}{\delta}v(1-v^{2})
\]
by the standard central difference scheme for the Laplacian, and forward
Euler method in time. Here $cp(x)=P_{\Gamma}x$ is the closest point
function. Instead of reinitializing $v$ to be constant-along-normal
after each time step (as in the original CPM), we reinitialize every
$h/2$ time interval (every 1000 explicit Euler steps).  {{}
The chosen time interval would be possible for a CPM that uses implicit
time stepping, but as this example shows, this choice may not be appropriate.
We refer the readers to \cite{macdonald_ruuth09} for the implicit
closest point method, and emphasize that the real performance of an
appropriately implemented implicit CPM may be different from what
is presented in this example.} The errors computed by two different
methods are shown in Figure \ref{fig:Allen-Cahn}. The oscillation
of the error computed by the modified CPM is due to the insufficient
frequency of reinitialization compared to the small time scales in
the dynamics. The large error near $t=1$ indicates the phase transition
of the gradient descent is incorrect. In the right subfigure in Figure
\ref{fig:Allen-Cahn}, we show the modified CPM solution at $t=0.9774975$,
right before the reinitialization step. In the center subplot of the
same Figure, we show the solution computed by the proposed method.
The modified CPM solution, before the last reinitialization, is not
constant-along-normal in the right top corner. One may see that in
such region, reinitialization can either set $v$ to back to the right
``phase'' value or to the opposite phase value, depending of the
developed pattern. Furthermore, since the transition time for Allen-Cahn
is very short, if reinitialization in a CPM algorithm is not applied
sufficiently frequently, the solution can lead to wrong gradient descent
flow and possibly different steady state pattern. 

\begin{figure}
\begin{centering}
\includegraphics{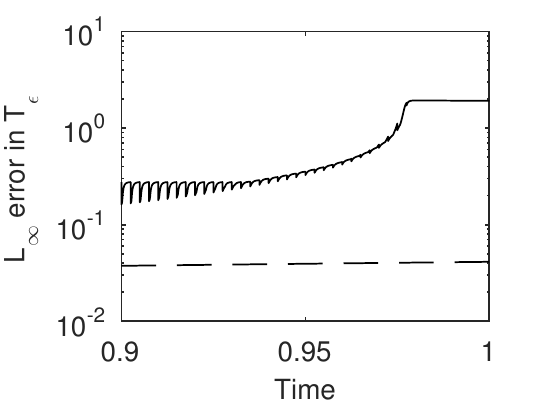}\includegraphics{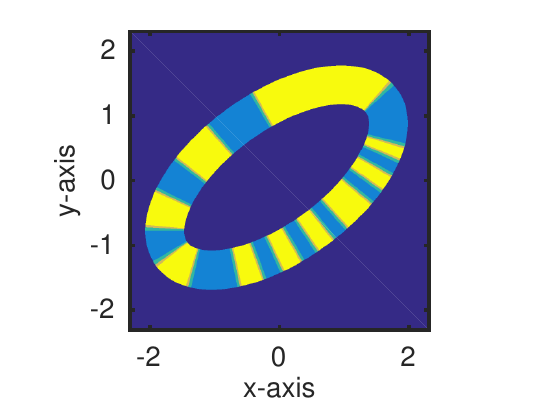}\includegraphics{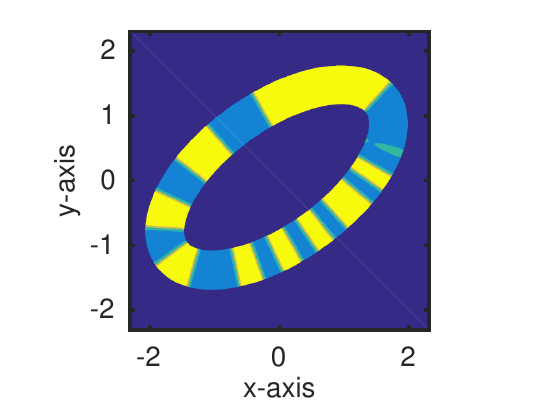}
\par\end{centering}
\caption{(Left) $L_{\infty}$-error in $T_{\epsilon}$ at different time step.
The dashed line is obtained by the proposed method and the solid line
is obtained by modified CPM. (Center) The snapshot of the solution
by the proposed method at $t=$ 0.9774975. (Right) The snapshot of
the solution by our modified CPM at the same time, before the last
reinitialization step. \label{fig:Allen-Cahn}}
\end{figure}

\subsubsection{Minimization of $p$-Laplacian on a torus with constraint\label{subsec:Minimization-of--Laplacian}}

Let $\Gamma$ be the torus $(\sqrt{x^{2}+y^{2}}-0.7)^{2}+z^{2}=0.3^{2}$
in $\mathbb{R}^{3}$. We can parametrize $\Gamma$ by $(x,y,z)=((0.7+0.3\cos\phi)\cos\theta,(0.7+0.3\cos\phi)\sin\theta,0.3\sin\phi),\theta,\phi\in[0,2\pi).$
Consider the energy function

\begin{equation}
I_{\Gamma}(u)=\frac{1}{p}\int_{\Gamma}|\nabla_{\Gamma}u|^{\,p}\,dS-\int_{\Gamma}f\,u\,dS.
\end{equation}
with constraint
\[
\|u\|_{2,\Gamma}=\int_{\Gamma}u^{2}\,dS=1.
\]
We use the method proposed in Section \ref{subsec:Constraints} to
solve the minimization problem of $I_{\Gamma}$ with the constraint.
We solve the gradient flow 

\[
v_{t}=J^{-1}\nabla\cdot(|\,A\nabla u\,|^{p-2}JA^{2}\nabla v)+\overline{f}+\lambda(t)v,
\]
on the narrowband $T_{\epsilon}$ of $\Gamma$ and the Lagrange multiplier
is chosen to be
\[
\lambda(t)=\frac{\int_{T_{\epsilon}}|\,A\nabla v\,|^{p}JdS-\int_{T_{\epsilon}}\overline{f}\,v\,JdS}{\int_{\Gamma}v^{2}J\,dS}.
\]
In the following examples, we choose the parameter $\mu$ to be $1$
and the kernel $K$ to be constant. The numerical simulations are
carried out on the grid nodes in $T_{\epsilon}\cap h\mathbb{Z}^{3}$
with $\epsilon=4h$, and $h$ is chosen to be $\frac{1}{25},\frac{1}{50}.$ 

We first choose $p=3$, $f=1$. In this case, the minimizer is $u=\frac{1}{c}$
with $c=\sqrt{\frac{84}{100}}\pi$, the square root of surface area
of $\Gamma$ and the minimal energy $I_{\Gamma}$ is $-c$. We use
two different initial conditions $u_{1}(\theta,\phi)=\frac{2}{c}\sin\theta\cos\phi$
and $u_{2}(\theta,\phi)=\frac{2}{c}\cos\theta\sin\phi$, $\theta,\phi\in[0,2\pi)$
on $\Gamma$ and extend them constant-along-normal to get initial
condition for extended equations. For time discretization, we use
forward Euler with $\Delta t=\frac{h^{2}}{10}$ to compute the solution
until $\|v(t+\Delta t)-v(t)\|_{T_{\epsilon},\infty}<10^{-7}.$ The
total energy error $J_{T_{\epsilon}}(v(t))+c$ and $L_{2}-$norm error
$|\frac{1}{2\epsilon}\int_{T_{\epsilon}}v^{2}(t)J\,dS-1|$ are shown
in Figure \ref{fig:p-Laplacian f1}. We can see the total energy converges
to the exact energy value $-c$ for both choices of initial conditions.
From the right subfigure in Figure \ref{fig:p-Laplacian f1}, we see
that the $L_{2}-$norm is almost conserved by the proposed method
and converges to $0$ as $h$ tends to $0$. 

For a nontrivial numerical example, we choose $p=3$, and $f=\cos(\theta+\phi)\sin(\phi)$.
For this problem, we do not have analytical form for the minimizer.
The total energy $J_{T_{\epsilon}}(v(t))$ and $L_{2}-$norm error
$|\frac{1}{2\epsilon}\int_{T_{\epsilon}}v^{2}(t)J\,dS-1|$ are shown
in Figure \ref{fig:p-Laplacian f2}. Again, we can see energy decays
to some steady state for both initial conditions and the $L_{2}-$norm
is almost conserved by the proposed method. 

\begin{figure}
\begin{centering}
\includegraphics{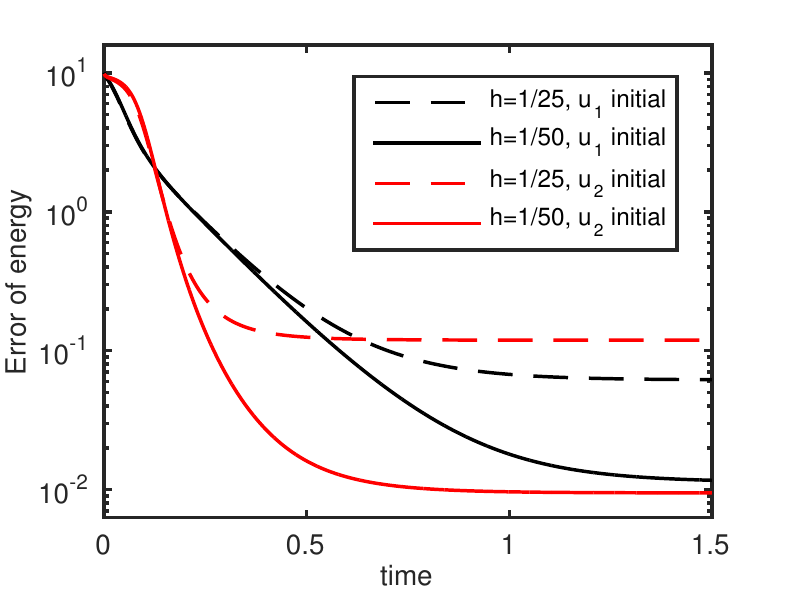}\includegraphics{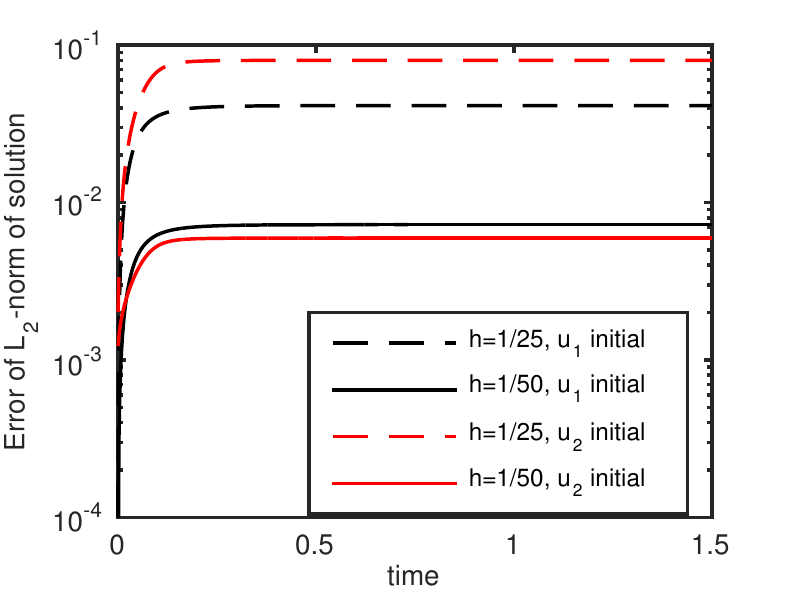}
\par\end{centering}
\caption{(Left) The total energy error $J_{T_{\epsilon}}[v]+c.$ (Right) The
error of $|\|v\|_{2,\Gamma}-1|$. We use $h=\frac{1}{25}$ and $h=\frac{1}{50}$
with two different initial condition $u_{1}$ and $u_{2}$. The dashed
and solid curves are for $h=\frac{1}{25}$ and $h=\frac{1}{50}$ respectively.
The black and red curves are using initial condition $u_{1}$ and
$u_{2}$ respectively.\label{fig:p-Laplacian f1}}
\end{figure}

\begin{figure}
\begin{centering}
\includegraphics{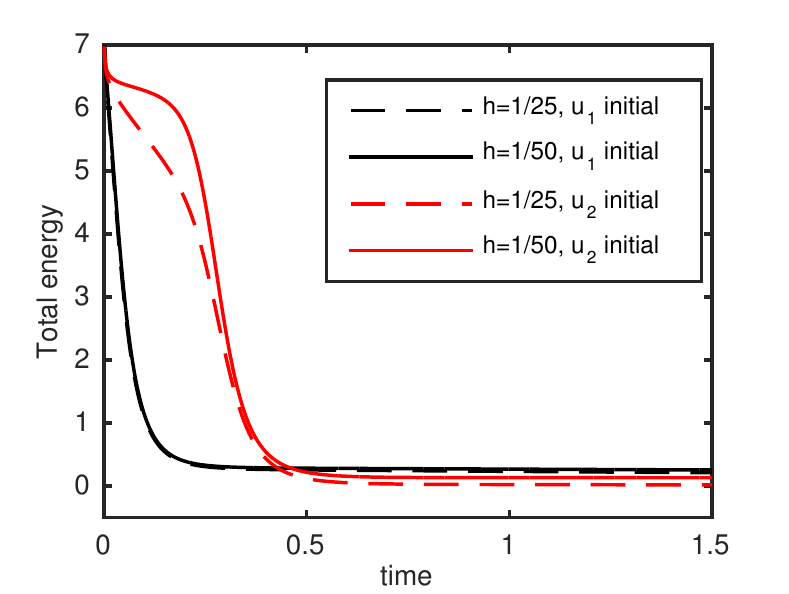}\includegraphics{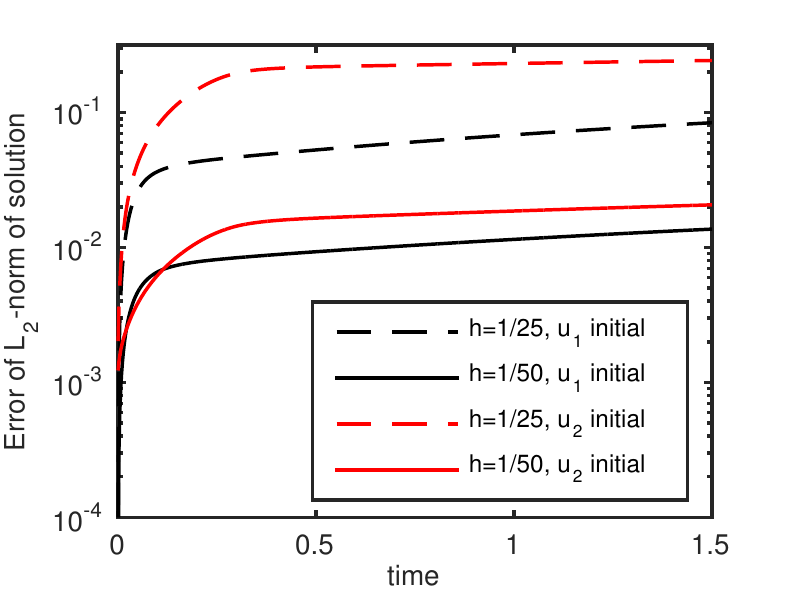}
\par\end{centering}
\caption{(Left) The total energy $J_{T_{\epsilon}}[v].$ (Right) The error
of $|\|v\|_{2,\Gamma}-1|$. We use $h=\frac{1}{25}$ and $h=\frac{1}{50}$
with two different initial condition $u_{1}$ and $u_{2}$. The dashed
and solid curves are for $h=\frac{1}{25}$ and $h=\frac{1}{50}$ respectively.
The black and red curves are using initial condition $u_{1}$ and
$u_{2}$ respectively.\label{fig:p-Laplacian f2}}
\end{figure}

\section{Least action principles\label{sec:Least-action-principles}}

We shall consider the model least action principle defined by 
\[
I_{\Gamma}[u]:=\int_{t_{1}}^{t_{2}}\int_{\Gamma}\left[\left(\frac{\partial u}{\partial t}\right)^{2}-|\nabla_{\Gamma}u|^{2}\right]dS(x)\,dt,
\]
which consists of the difference between the kinetic energy and the
potential energy. Taking the variational derivative, the resulting
Euler-Lagrange equation is an analogy of the usual second order wave
equation, but defined on $\Gamma$. 

Correspondingly, we shall study the extensions of $I_{\Gamma}$ in
the following form
\begin{equation}
\tilde{I}_{T_{\epsilon}}[v]:=\int_{t_{1}}^{t_{2}}\int_{T_{\epsilon}}\left[\left(\frac{\partial v}{\partial t}\right)^{2}-|A(z;\mu)\nabla u|^{2}\right]K_{\epsilon}(d_{\Gamma}(z))J(z)dz\,dt.\label{eq:saddle_point_energy}
\end{equation}
Recall that $A(z;\mu):=\left(\sigma_{1}^{-1}\mathbf{t}_{1}\otimes\mathbf{t}_{1}+\sigma_{2}^{-1}\mathbf{t}_{2}\otimes\mathbf{t}_{2}+\mu\mathbf{n}\otimes\mathbf{n}\right).$
Hence the essential questions are about the well-posedness of the
extended problem, with or without the additional component $(\frac{\partial u}{\partial n})^{2}$
in the potential energy of the system, and stability of the constant-along-normal
solutions. 

  {Notice that unlike elliptic equations, the zero Neumann
boundary condition is not the natural boundary condition for the variational
problem involving energy $\tilde{I}_{T_{\epsilon}}$ defined in \eqref{eq:saddle_point_energy}.
The natural boundary condition is the radiation boundary conditions
(add citation). However, the constant-along-normal solutions usually
do not satisfy the radiation boundary conditions.}

\subsection{Study of a model problem}

Consider $\Gamma$ to be the unit interval on the $x$-axis with periodic
boundary conditions at $x=0$ and $1$, and $T_{\epsilon}=[0,1)\times(-\epsilon,\epsilon).$
Thus the Euler-Lagrange equation of $\tilde{I}_{T_{\epsilon}}$ satisfies
\begin{equation}
v_{tt}=v_{xx}+\mu v_{yy},\,\,\,(\mu\geq0)\label{eq:model-wave}
\end{equation}
 We look for periodic solutions 
\begin{equation}
v(x+1,y,t)=v(x,y,t)\label{eq:wave-periodic-BC}
\end{equation}
 satisfying the \emph{constant-along-normal} initial conditions 
\begin{equation}
\begin{cases}
v(x,y,0)=v_{0}(x),\\
v_{t}(x,y,0)=v_{t,0}(x),
\end{cases}\label{eq:wave-ICs}
\end{equation}
and the Neumann boundary condition 
\begin{equation}
\frac{\partial v}{\partial y}(x,\pm\epsilon,t)=0.\label{eq:wave-Neumann}
\end{equation}
For $\mu=0,$ we have a family of one-dimensional wave equation, coupled
only through the initial conditions. It is obvious that for $\mu\ge0,$
the initial-boundary-value problem \eqref{eq:model-wave}-\eqref{eq:wave-Neumann}
is well-posed and stable. Furthermore, it is easy to show that the
solutions starting with the initial conditions \eqref{eq:wave-ICs}
will stay \emph{constant-along-normal}, i.e. $v_{y}(x,y,t)=0$ for
all time. 

  {In this case, the natural boundary condition is $v_{t}\pm v_{y}=0$
on $y=\pm\epsilon$. For this model problem, the constant-along-normal
solutions are superposition of plane waves of the form $e^{i(kx\pm kt)}$.
Obviously, the planes wave solutions do not satisfy the natural boundary
conditions unless they are constants.}

In the following, we shall look into the effect of perturbations that
might be introduced by discretization. 

\subsubsection{Perturbation in the initial conditions}

We consider the solutions of the form
\[
u(x,y,t)=f(x+t)+g(x-t)+C(y,t),\,\,\,C(y,0)=C_{t}(y,0)=0.
\]
For $\mu\ge0$, $C(y,t)$ can only be zero due to the special form
of the initial conditions. Now, suppose that the initial condition
is not exactly constant along the normals of $\Gamma$, the equation
with $\mu=0$ will admit solutions that grow linearly in time. 

To see this, consider $\frac{\partial v}{\partial y}(x,y,0)=0$ and
\[
\frac{\partial}{\partial y}v_{t}(x,y,0)=-\delta\sin\frac{\pi y}{\epsilon}\,\,\,\text{for some }\delta\in\mathbb{R}.
\]
This means that the initial condition of $v_{t}$ has a sinusoidal
perturbation satisfying the Neumann boundary condition. Then $C(y,t)=\frac{\delta\epsilon}{\pi}\cos\frac{\pi y}{\epsilon}\cdot t.$
Computationally, if both $\delta$ and $\epsilon$ are of order $h$,
with $h$ the grid spacing, this initial perturbation will introduce
only $O(t\cdot h^{2})$ difference in the solutions. 

With $\mu>0,$ the perturbation in the initial conditions will result
in modes that propagate towards the boundaries and then be reflected
back into the domain. Nevertheless, since the total energy is conserved,
we expect that such perturbation will stay controlled. 

\subsubsection{Perturbation in the boundary conditions}

In general, numerical discretization of the Neumann boundary condition
\eqref{eq:wave-Neumann} will inevitably introduce some tangential
perturbations that could be modeled by: 
\begin{align}
\frac{\partial v}{\partial y}(x,\epsilon,t) & =\alpha\frac{\partial^{k}v}{\partial x^{k}}(x,\epsilon,t),\,\,\,\alpha\in\mathbb{R},k\ge1,\label{eq:perturbed-Neumann-BC}
\end{align}
and
\begin{equation}
\frac{\partial v}{\partial y}(x,-\epsilon,t)=0.\label{eq:zero-Neumann-at-y=00003Depsilon}
\end{equation}
Consider the following two 4th order boundary closures for the grid
nodes $(ih,(N+1)h)$:

\[
u_{i,N+1}:=\frac{2}{3}(u_{i-1,N-1}+u_{i+1,N-1})-\frac{1}{6}(u_{i-2,N-1}+u_{i+2,N-1}),
\]
 leading to 
\[
v_{y}\approx\frac{1}{6}h^{4}u_{xxxx},
\]
 and
\[
u_{i,N+1}:=\frac{1}{4}u_{i-1,N-1}+\frac{3}{2}u_{i+1,N-1}-u_{i+2,N-1}+\frac{1}{4}u_{i+3,N-1},
\]
 leading to a perturbation to the zero Neumann condition with the
leading order term of a different sign:
\[
v_{y}\approx-\frac{1}{4}h^{4}u_{xxxx}.
\]

A natural question is whether such perturbation will render the resulting
initial-boundary-value problem ill-posed? We follow closely the theory
developed in \cite{gustafsson-kreiss-oliger}, and particularly in
\cite{kreiss-petersson-ystrom2004}. We use Laplace transform in time
and Fourier transform in the $x$-variable to study the well-posedness
of the perturbed problem \eqref{eq:model-wave}-\eqref{eq:wave-ICs}
and \eqref{eq:perturbed-Neumann-BC}-\eqref{eq:zero-Neumann-at-y=00003Depsilon}.
Due to linearity of the problem, we may analyze by a single mode:
\begin{equation}
\hat{v}(\omega,y,s)=e^{st}e^{i\omega x}\phi(y).\label{eq:laplace-fourier-mode}
\end{equation}
Suppose that there exists $s_{0}\in\mathbb{C}$ with $\text{Re }s_{0}>0$
and $\omega_{0}\in\mathbb{Z}$ such that $\hat{v}$ solves \eqref{eq:model-wave}-\eqref{eq:wave-ICs}
and \eqref{eq:perturbed-Neumann-BC}-\eqref{eq:zero-Neumann-at-y=00003Depsilon}. 

We first remark that even with constant-along-normal initial conditions
\eqref{eq:wave-ICs}, $v_{y}$ will not remain $0$ in time, due to
the perturbed boundary condition \eqref{eq:perturbed-Neumann-BC},
unless $v$ is a constant. Therefore, we cannot just consider $\phi(y)$
being a constant. Plugging \eqref{eq:laplace-fourier-mode} into the
equation, we obtain
\[
s^{2}\phi=-\omega^{2}\phi+\mu\phi_{yy}.
\]
Without loss of generality we analyze $\mu=1$ for the case $\mu>0$.
We have the eigenvalue problem 
\begin{equation}
\phi_{yy}=(s^{2}+\omega^{2})\phi,\label{eq:eigen-value problem}
\end{equation}
with the boundary conditions 
\begin{equation}
\phi_{y}(\epsilon)=\alpha(i\omega)^{k}\phi(\epsilon),\label{eq:fourier-transform of vy=00003Davx}
\end{equation}
and
\[
\phi_{y}(-\epsilon)=0.
\]
\[
\phi(y)=c_{1}e^{\kappa y}+c_{2}e^{-\kappa y},\,\,\,\kappa:=(s^{2}+\omega^{2})^{\frac{1}{2}},
\]
where the square root is taken on the complex plane with the branch
$-\pi<\arg z\le\pi.$ Solutions with $\kappa=0$ are stable, so we
will not discuss such a case. If $\kappa\neq0$, from the boundary
conditions, we have 
\[
c_{1}=c_{2}e^{2\kappa\epsilon},
\]
and 
\begin{equation}
(\kappa-\alpha(i\omega)^{k})e^{4\kappa\epsilon}=(\kappa+\alpha(i\omega)^{k}).\label{eq:perturbed-constraint-on-k-omega}
\end{equation}

\paragraph*{$k$ is an even positive integer.}

In this case, \eqref{eq:perturbed-constraint-on-k-omega} is equivalent
to $\kappa\tanh(2\kappa\epsilon)=\tilde{\alpha}\omega^{k}$, where
$\tilde{\alpha}=\alpha i^{k}$. This equation admits only pure imaginary
solution $\kappa$ if $\tilde{\alpha}\leq0$; and it has at least
one real solution $\kappa$ if $\tilde{\alpha}>0$. Thus we conclude
that $s=(\kappa^{2}-\omega^{2})^{\frac{1}{2}}$ is a pure imaginary
number if $\tilde{\alpha}<0$, and in this case the equation is stable.
On the other hand, we have $s=(\kappa^{2}-\omega^{2})^{\frac{1}{2}}>0$
if $\tilde{\alpha}>0$. Therefore the equation is unstable when $\tilde{\alpha}>0$.

\paragraph{$k$ is an odd positive integer. }

In this case, \eqref{eq:perturbed-constraint-on-k-omega} is equivalent
to 
\begin{equation}
\frac{\tau+bi}{\tau-bi}=e^{\tau},\label{eq:stability equation}
\end{equation}
where $\tau=4\epsilon\kappa,b=4\epsilon\alpha i^{k-1}\omega^{k}$.
Suppose $\tau=u+vi$, then by equating the modulus in both sides of
\eqref{eq:stability equation} we have

\[
\frac{u^{2}+(v+b)^{2}}{u^{2}+(v-b)^{2}}=e^{2u}.
\]
Hence we have $v=f(u)=\frac{b(e^{2u}+1)\pm\sqrt{4b^{2}e^{2u}-u^{2}(e^{2u}-1)^{2}}}{e^{2u}-1}$
and $v$ is real if $u\sinh u\leq|b|$. Notice that $\omega=2n\pi$
can be arbitrary large so can $|b|$. To show \eqref{eq:stability equation}
admits a solution, we check that

\[
\frac{u+(v+b)i}{u+(v-b)i}=e^{2u}e^{vi}\iff\frac{u+(f(u)+b)i}{u+(f(u)-b)i}=e^{2u}e^{f(u)i}.
\]
We only need to show there exists $u$ such that the arguments of
both sides in the above equation are equal. With out loss of generality,
we assume $b>0$. Let $u^{*}>0$ satisfies $u^{*}\sinh u^{*}=b$.
Then $f(u^{*})=u^{*}\cosh(u^{*})$ and $arg(\frac{u^{*}+(f(u^{*})+b)i}{u^{*}+(f(u^{*})-b)i})\rightarrow arg(e^{u^{*}}i)=\frac{\pi}{2}$
as $|b|\rightarrow\infty.$ Therefore for large enough $b$ and $u$
close to $u^{*}$ enough, $arg(\frac{u+(f(u)+b)i}{u+(f(u)-b)i})\simeq\frac{\pi}{2}$.
Since $f(0)=0$ and $f(u)$ is increasing, by continuity of $f(u)$,
for large enough $\omega,$ there always exists $\tau=u+f(u)i$ that
satisfies \eqref{eq:stability equation}. Therefore $s=\sqrt{\frac{\tau^{2}}{16\epsilon^{2}}-\omega^{2}}$
with positive real part satisfies the eigenvalue problem. This means
that the perturbed boundary condition admits solutions that exponentially
increase in time and have boundary layers of near $\partial T_{\epsilon}$. 

\subsubsection{General positive values of $\mu$. }

$s$ is replaced by $s/\sqrt{\mu}$ in the analysis above, and the
unstable modes become
\[
\hat{v}(x,y,t)\sim e^{\mu st}e^{i\omega x}(c_{1}e^{\mu\kappa y}+c_{2}e^{-\mu\kappa y}).
\]
 We see that with small values of $\mu$, the instability can be controlled. 

For $\mu=0,$ we have 
\[
(s^{2}+\omega^{2})\phi(y)=0.
\]
 Non-trivial solutions requires that $s=\pm i\omega$. We then look
for $\phi(y)$ which satisfies 
\[
\phi_{y}(\epsilon)=\alpha(i\omega)^{k}\phi(\epsilon),\,\,\,\phi_{y}(-\epsilon)=0.
\]
 The set of $C^{2}$ functions which satisfy the two conditions are
not bounded. Thus we have to estimate the solution from the initial
conditions. 

\begin{figure}
\begin{centering}
\includegraphics{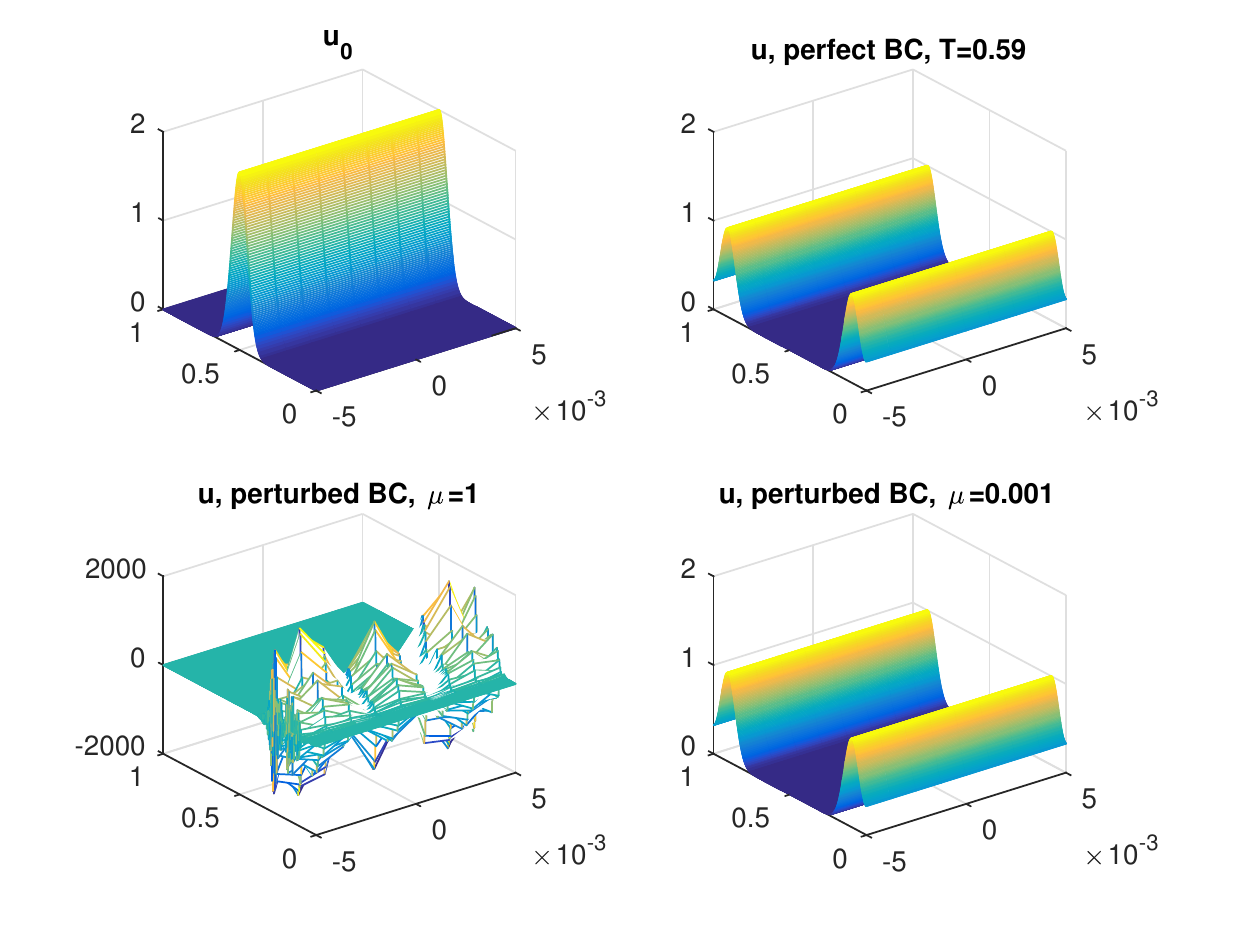}
\par\end{centering}
\caption{\label{fig:Unstable-wave-solutions}Unstable wave solutions by a 3rd
order boundary closure, and stabilization by using smaller $\mu$.}
\end{figure}

Figure~\ref{fig:Unstable-wave-solutions} demonstrates the instability
analyzed above as well as the simple stabilization by using a small
$\mu.$ Figure~\ref{fig:Unstable-wave-solutions-4th-order-BCs} demonstrates
a numerical simulation using these two boundary closures with different
values of $\mu$. 

\begin{figure}
\begin{centering}
\includegraphics{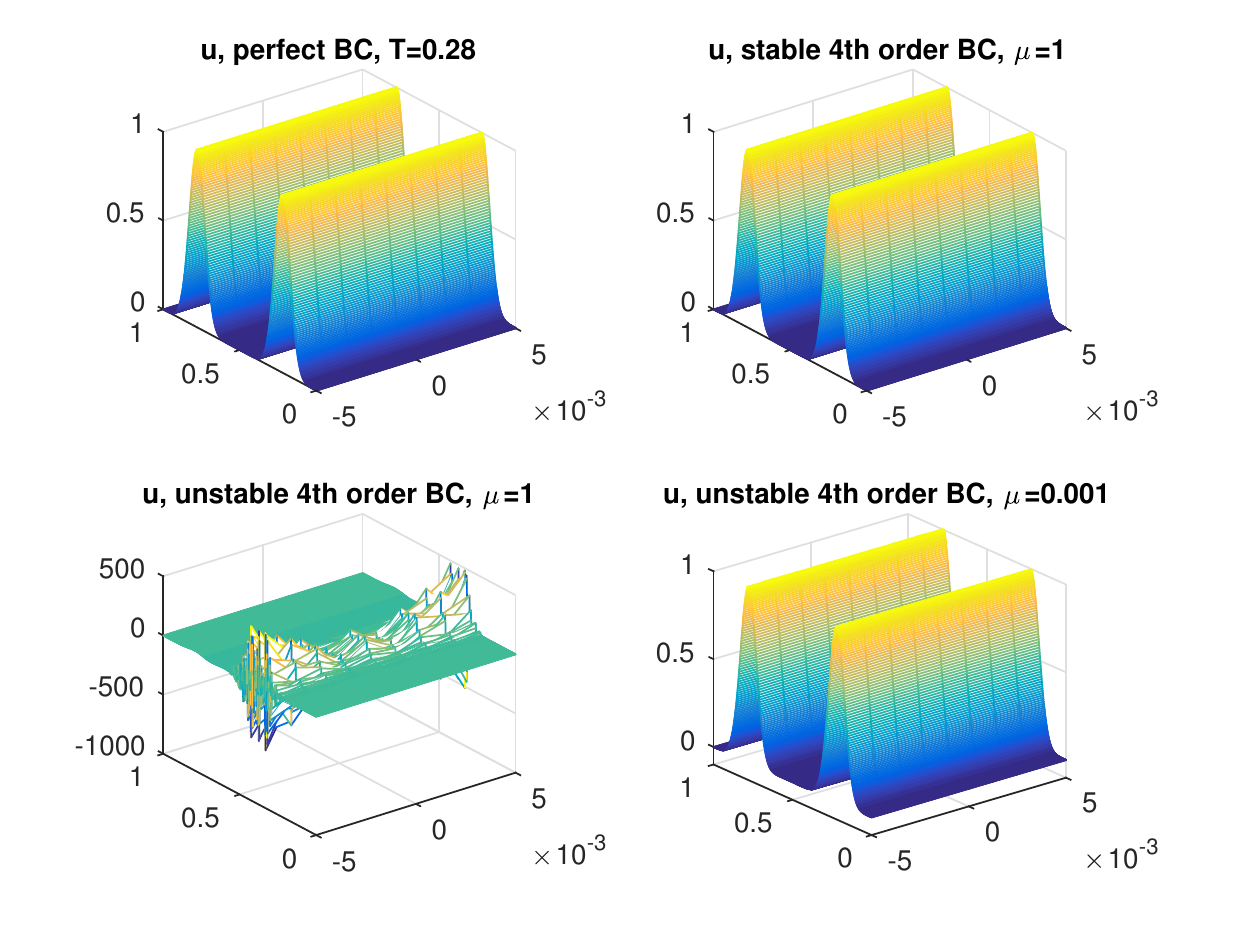}
\par\end{centering}
\caption{\label{fig:Unstable-wave-solutions-4th-order-BCs} Wave solutions
with stable and unstable 4th order boundary closure, and stabilization
by using smaller $\mu$.}
\end{figure}

\subsection{Perturbed boundary condition and stability}

The mode analysis above shows that the instability exists in the normal
derivatives of the solutions. We may explain the onset of instability
as follows: waves are being bounced back and forth in between the
two boundaries of $T_{\epsilon}$, and each time reflection takes
place, the solution may loss some regularity, and after a few reflections,
the accumulated instability dominates the computed system. In fact,
the loss of regularity can be read off from the perturbed boundary
conditions: the normal derivative of the reflected solution is set
to the derivative of the part of the solution that caused the reflection.
We can also understand at a heuristic level that without dissipation,
such instability is hard to avoid for more general domains with curved
boundaries (where as for dissipative systems, such problem is less
prone to happen). 

We also see this mechanism directly from the model equation and conditions
satisfied by $w=v_{y}$:{\small{} }
\begin{equation}
w_{tt}=w_{xx}+\mu w_{yy},\,\,\,(\mu\geq0),\label{eq:model-wave-1}
\end{equation}
satisfying the periodic condition in $x$:
\begin{equation}
w(x+1,y,t)=w(x,y,t),\label{eq:wave-periodic-BC-1}
\end{equation}
 the Neumann boundary conditions 
\begin{align}
w(x,\epsilon,t) & =\alpha\frac{\partial^{k}v}{\partial x^{k}}(x,\epsilon,t),\label{eq:wave-Neumann-1}\\
w(x,-\epsilon,t) & =0,
\end{align}
and the \emph{constant-along-normal} initial conditions for $v$:
\begin{equation}
\begin{cases}
w(x,y,0)=0,\\
w_{t}(x,y,0)=0.
\end{cases}\label{eq:wave-ICs-1}
\end{equation}
We first see that by reducing the size of $\mu$, we delay the propagation
of the boundary perturbation into the domain. Due to the special form
of the initial conditions, we expect that the instability reflected
in the discretized system comes out after constant multiply of discrete
time steps. 

For general curved boundaries, we cannot rely on diminishing the size
of $\mu$ to prolong the onset of the instability. Suppose that we
set $\mu=0$ in the equation for $v$, and analytically, with the
constant-along-normal initial data, the solution will constants of
wave that propagates only tangentially to $\Gamma_{\eta}.$ However,
in the discretized system, a wave that is tangential to the boundary
at a grid node at time $t^{n}$ will not be exactly tangential when
it arrives at a neighboring grid node at a later time. And there,
a reflection of this wave will take place. The following numerical
simulations verify the discussion above. 
\begin{example}
\label{exp: Stability-against-perturbation}(Stability against perturbation
in the initial conditions) Let $\Gamma$ be the ellipse with major
axis equal to 4 and minor axis equal to 2 in $\mathbb{R}^{2}$ as
described in Section \ref{subsec:Comparison-of-decaying}. We consider
the wave equation on $\Gamma:$ $u_{tt}=u_{ss}$, where $s$ is the
arc length of $\text{\ensuremath{\Gamma}}$ and the exact solution
is given by $u(s,t)=\sin(\frac{2\pi(s-t)}{L})$. The extended wave
equation on $T_{\epsilon}$ is
\end{example}
\begin{equation}
v_{tt}=\sigma^{-1}\nabla(A(\sigma;\mu)\nabla v),\label{eq:wave_ellipse}
\end{equation}
where $A(\sigma;\mu)=\sigma^{-1}{\bf t}\otimes{\bf t}+\mu{\bf n}\otimes{\bf n}$.
We use leapfrog scheme in time and second order central difference
scheme in space to discretize \eqref{eq:wave_ellipse}. See Appendix
for detail of discretization. The boundary nodes are interpolated
by bi-cubic interpolating polynomials. To see the effect of being
constant-along-normal or not, we perturb the initial condition in
the numerical simulation. The initial condition is given by
\[
v(s,\eta,0)=(C+(1-C)\cos\frac{\pi\eta}{2\epsilon})\sin(\frac{2\pi s}{L}),v(s,\eta,\Delta t)=(C+(1-C)\cos\frac{\pi\eta}{2\epsilon})\sin(\frac{2\pi(s-\Delta t)}{L}).
\]
The parameter $C$ is to control the intensity of the perturbed initial
condition away from being constant-along-normal. When $C=1$, the
initial data is constant-along-normal; The smaller $C$ the bigger
variance along the normal direction. We use $h=\frac{1}{100},\frac{1}{200},\frac{1}{400}$,
$\epsilon=3h,6h,0.1$, $C=1,0.9,0.8$ and $\Delta t=\frac{h}{10}$
to test how fast the instability exhibits. We compute the solution
until time step $t_{blow}$ such that $\|v(t_{blow})\|_{T_{\epsilon},\infty}>1.5$
and use $t_{blow}$ as the indicator for instability. From Figure
\ref{fig:Blowup-time}, we see that it makes sense to use a smaller
$\mu$ so that the instability occurs later. However, having it too
small or zero will not be beneficial. Numerical evidence shows that
the best choice of $\mu$ is about size of $h$. We also notice that
the wider of bandwidth $\epsilon$ can delay the blowup time for some
cases but not efficiently. Being far away from constant-along-normal
produces the instability in a very short time. The smaller grid size
$h$ causes the smaller blowup time. Numerically it shows that the
blowup time is of order $O(h).$ In other words, the instability accumulates
in each iteration independent of grid size $h$. 

\begin{figure}
\begin{centering}
\includegraphics{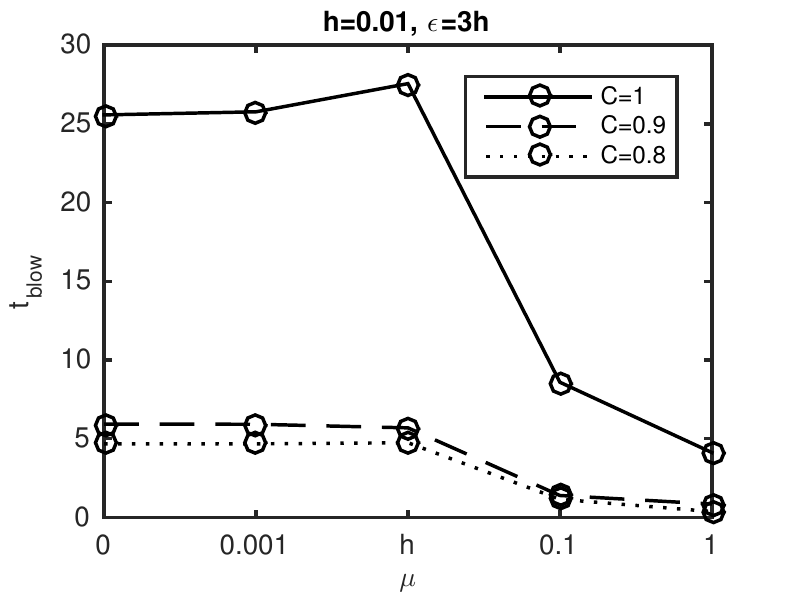}\includegraphics{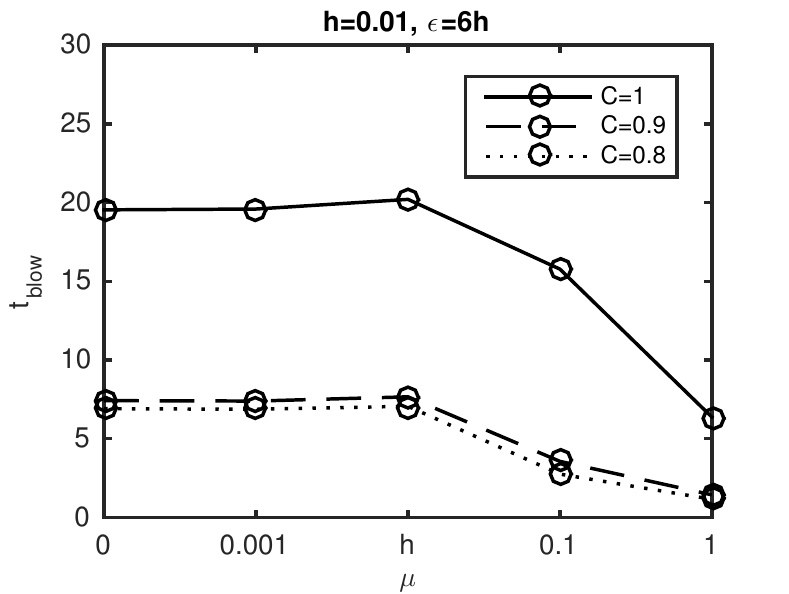}
\par\end{centering}
\begin{centering}
\includegraphics{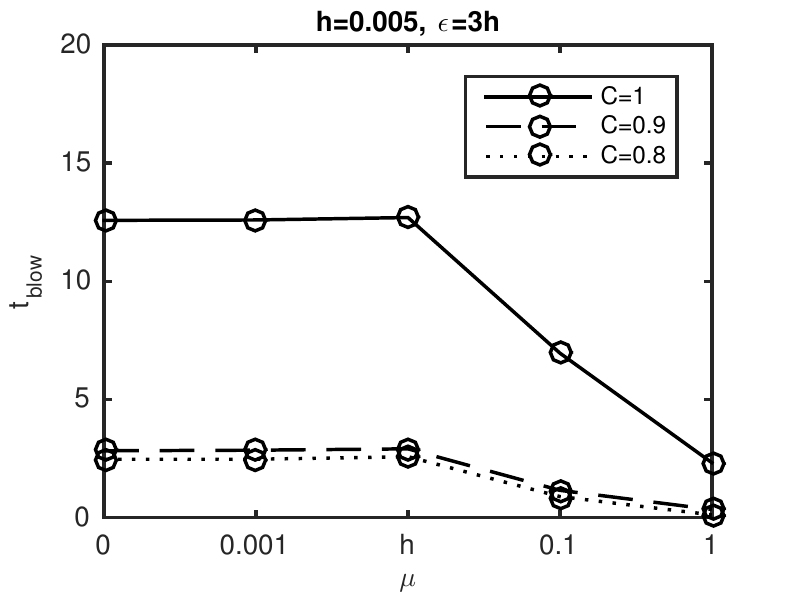}\includegraphics{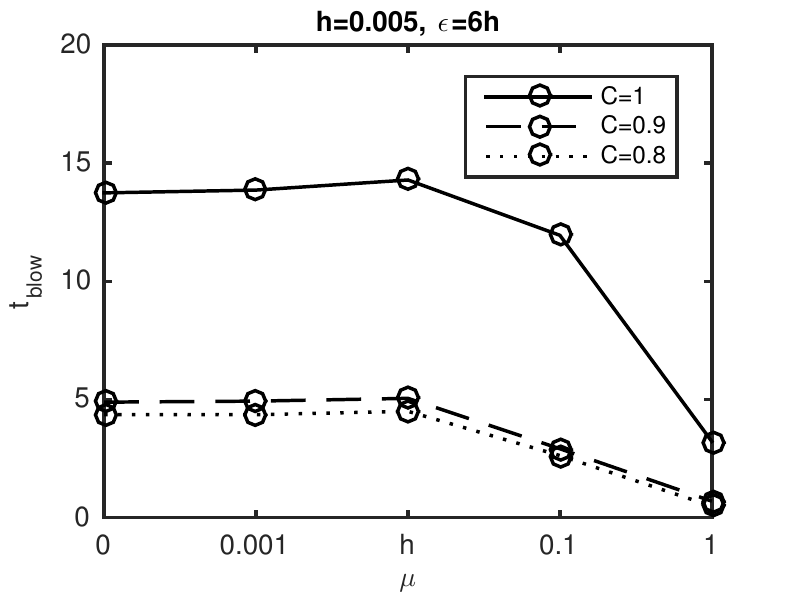}
\par\end{centering}
\begin{centering}
\includegraphics{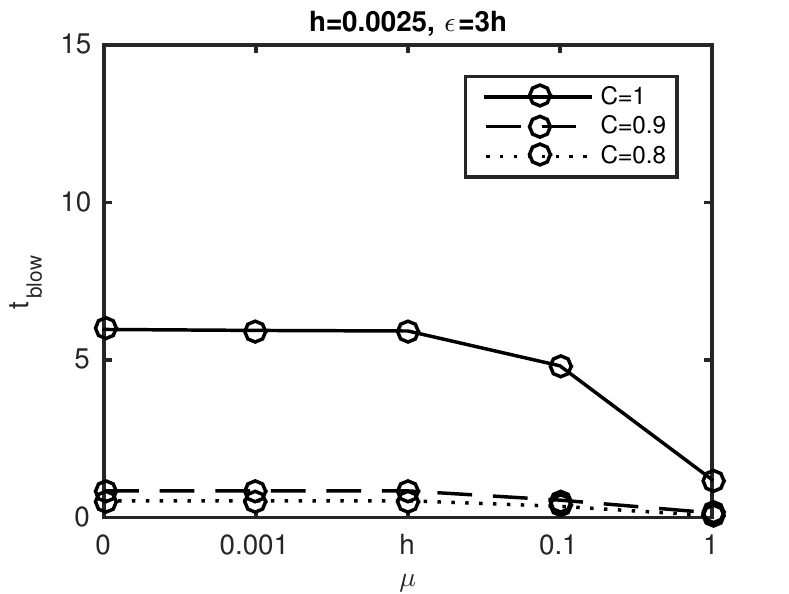}\includegraphics{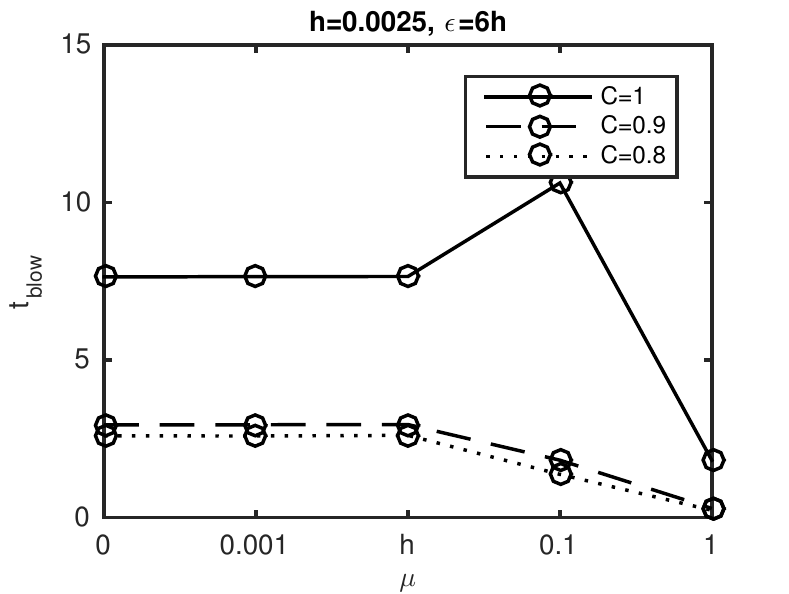}
\par\end{centering}
\caption{Blowup time $t_{blow}$ for different grid size $h$ and bandwidth
$\epsilon$. The larger perturbation in normal direction produces
the instability faster. The smaller $\mu$ can delay the onset of
instability. The blowup time $t_{blow}$ scales like $O(h)$.\label{fig:Blowup-time}}
\end{figure}

\begin{figure}
\begin{centering}
\includegraphics{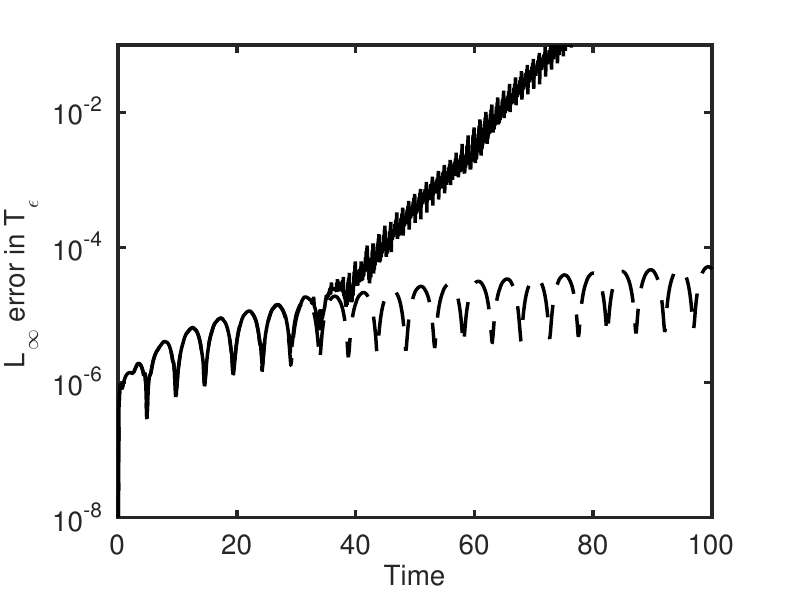}\includegraphics{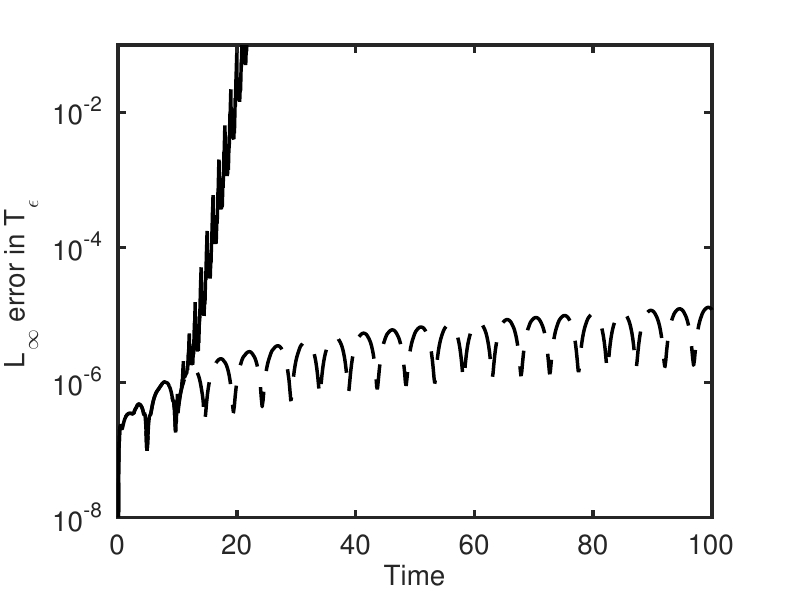}
\par\end{centering}
\caption{Reinitialization strategy to stabilize the wave equation. We use $h=1/200$
in the left figure and $h=1/400$ in the right figure. The dashed
and solid curves are obtained by reinitializing per 0.1 and 1 time
step.\label{fig:Reinitialization}}
\end{figure}

\subsection{Stabilization strategies and examples}

\subsubsection{Reinitialization}

For general initial data, the unstable modes will dominate the solution
also instantaneously. However, for constant-along-normal initial data
considered in our particular problems, particularly with smaller values
of $\mu$, the unstable modes seem to take longer time to become dominant. 

Therefore, we propose to stabilize the computations by \emph{reinitializing
the computed solutions }periodically. By reinitializing a function
$f_{0}$, we mean to create a new function $f$, which is the constant-along-normal
extension of $f_{0}|_{\Gamma}$. Such reinitialization can be done
easily by applying the boundary closure strategy to every inner node,
projecting them onto $\Gamma$. It can also be done easily by solving
the constant-extension PDE, used in the level set method, see e.g.
 \cite{LevelSet_OsherFedkiw,Redistance_ChengTsai}. In a ``Closest
Point Method'', this step is a mandatory part of every discrete time
step.

To demonstrate the strategy, we use the same setting as in Example
\ref{exp: Stability-against-perturbation}. We reinitialize the solution
per 0.1 and 1 time unit (or equivalently per $h^{-1}$ and $10h^{-1}$
discrete time steps). In this experiments, we test for $h=\frac{1}{200},\frac{1}{400}$
and compare the $L_{\infty}$-error $\|v(s,\eta,t)-\sin(\frac{2\pi(s-t)}{L})\|_{\infty,T_{\epsilon}}$
. The results are shown in Figure \ref{fig:Reinitialization}. The
solid curves are obtained by reinitializing per 0.1 time unit and
the dashed curve are obtained by reinitializing per 1 time unit. We
see that the solutions is stable for much longer time after reinitialization.
However, if the reinitialization is not frequently enough, the instability
accumulates after certain iterations and the solution is unstable.

\subsubsection{``Dissipative'' regularization}

Another way to stabilize the solution is to follow and adapt the idea
proposed in \cite{kreiss-petersson-ystrom2004}. Consider our model
problem in the strip: 
\[
v_{t}=v_{xx}+\mu v_{yy}-\alpha v_{txxxx},\:\,\forall(x,y)\in[0,1)\times(-\epsilon,\epsilon),\,\alpha>0,
\]
with the perturbed boundary condition

\[
v_{y}(x,-\epsilon,t)=0,\,v_{y}(x,\epsilon,t)=av_{xxxx}.
\]
In the same fashion as employed earlier in this section, it is proven
in Section 6 and Appendix B of \cite{kreiss-petersson-ystrom2004},
that with this regularization and a suitable $\alpha$, the resulting
initial boundary value problem is well-posed.

For general setups, the regularization will correspond to higher order
tangential derivatives of the solution, and is not very convenient
to discretize. Therefore, we propose to add to the equation for $v$
an isotropic version of this regularization that involve only the
similar partial derivatives along the coordinate directions. This
means, in two dimensions, we shall modify the PDE by 
\[
v_{tt}=\sigma^{-1}\nabla\cdot(\sigma^{-1}\nabla v)-\alpha h^{3}(v_{txxxx}+v_{tyyyy})
\]
and discretize it on Cartesian grids by
\[
\frac{v_{i,j}^{n+1}-2v_{i,j}^{n}+v_{i,j}^{n-1}}{\Delta t^{2}}=\mathcal{A}_{i,j}v_{ij}^{n}-\alpha h^{3}(D_{+}^{x}D_{-}^{x})^{2}\left(\frac{v_{i,j}^{n}-v_{i,j}^{n-1}}{\Delta t}\right)-\alpha h^{3}(D_{+}^{y}D_{-}^{y})^{2}\left(\frac{v_{i,j}^{n}-v_{i,j}^{n-1}}{\Delta t}\right),
\]
 where $\alpha\le\frac{c}{4}$, $c$ is the maximum wave speed and
$\mathcal{A}_{i,j}v_{i,j}^{n}$ is the discretization of $\sigma^{-1}\nabla\cdot(\sigma^{-1}\nabla v)$.

To demonstrate the effect of stabilization, we use the same setting
in Example \ref{exp: Stability-against-perturbation}. We add the
stabilization terms and use $\alpha=\frac{1}{5}$ with $h=\frac{1}{100},\frac{1}{200},\frac{1}{400}$
and $\Delta t=\frac{h}{10}$. We compare the $L_{\infty}$-error $\|v(s,\eta,t)-\sin(\frac{2\pi(s-t)}{L})\|_{\infty,T_{\epsilon}}$
for different $\mu$ and the results are shown in Figure \ref{fig:stable}.
Notice that, to compute $(D_{+}^{x}D_{-}^{x})^{2}$ and $(D_{+}^{y}D_{-}^{y})^{2}$
terms, we need more ghost nodes. We see the solution becomes stable
for very long time $T\ge60$. Moreover, the solution has the same
magnitude $L_{\infty}$-error as the solution without stabilization
terms in the short time. This means the stabilization does not compromise
the accuracy but regain the stability for $\mathcal{O}(1)$ time step.
Also the results show the importance to choose free parameter $\mu$
close to $\mathcal{O}(h)$. The solution has smaller error for smaller
$\mu$. 

\begin{figure}
\begin{centering}
\includegraphics{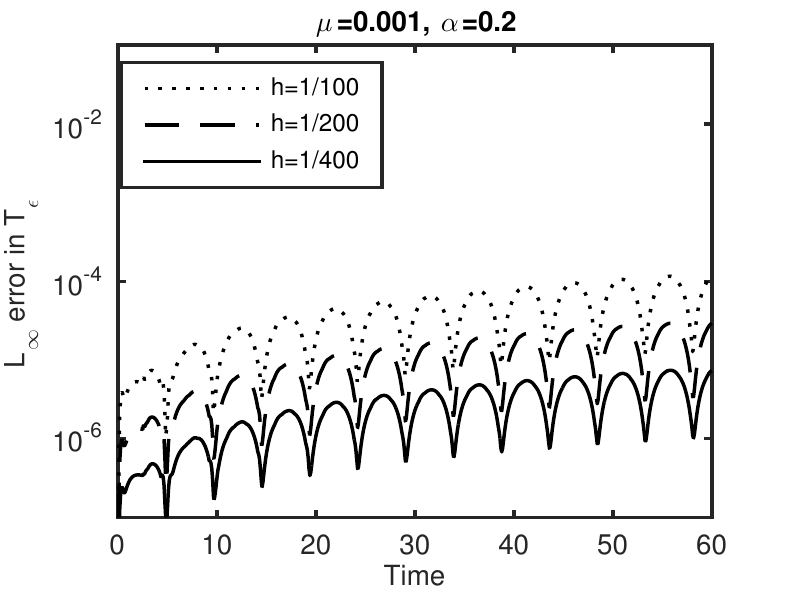}\includegraphics{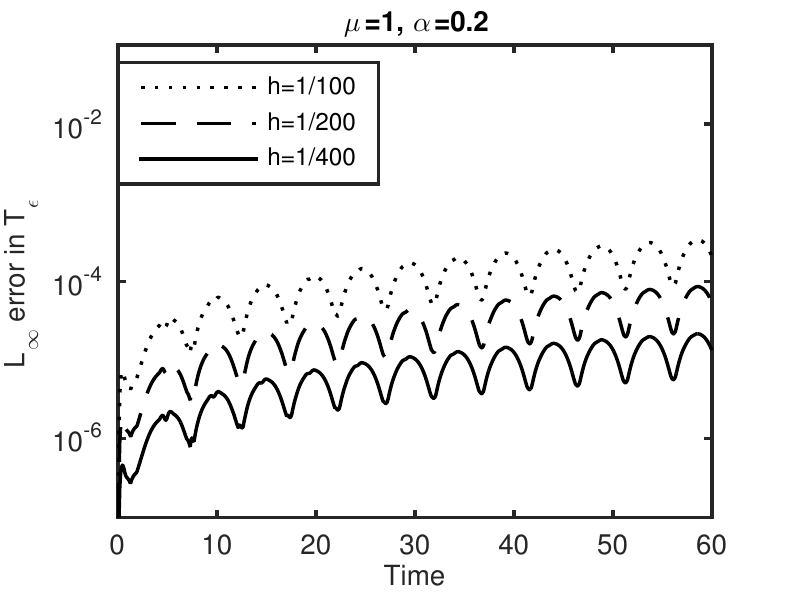}
\par\end{centering}
\begin{centering}
\includegraphics{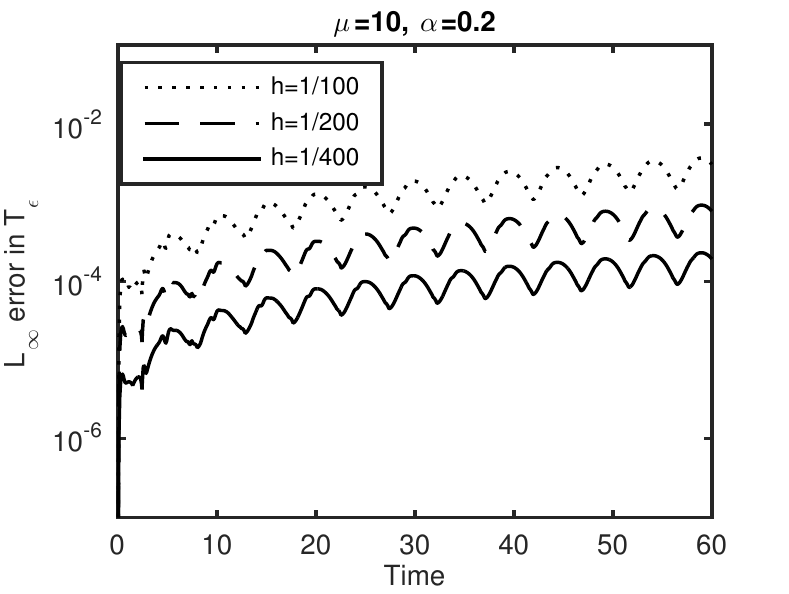}
\par\end{centering}
\caption{The $L_{\infty}$-error of solutions for wave equations with stabilization
terms.\label{fig:stable}}
\end{figure}

\section{Summary}

In this paper, we derive extensions of a class of integro-differential
operators defined on smooth closed manifolds. The extended operators
are defined for functions defined on thin narrowbands around the manifolds.
The main objective is to provide a formulation that allows for construction
of simple and accurate numerical algorithms that solve for the Euler-Lagrange
equations of the functionals defined by these integro-differential
operators, especially for the applications in which the manifolds
are defined by the distance function or closest point mapping to the
manifolds. 

What distinguishes this work from other existing level set or closest
point methods is that fact that our formulation solves the Euler-Lagrange
equations of the extended, volumetric integrals, with the corresponding
natural boundary condition on the boundaries of the narrowbands. We
investigate how the extensions can be made to guarantee a strict equivalence
between the solutions of the extended and the original problems. As
a result, eigenvalue problems involving the Laplace-Beltrami operators
can easily be computed. Together with our mathematical formulation,
we propose a simple boundary closure procedure that can be adapted
to a variety of numerical methods. We further study the stability
and well-posedness of a few model (initial)-boundary-value problems.
We discovered that hyperbolic problems require additional stabilization
to curb the effect of unstable mode originated from the approximation
of boundary conditions. We propose two strategies to stabilize the
numerical algorithms. One involves adding higher order ``dissipative''
terms to the PDEs, and the other one involves ``reinitializing''
the solutions of the PDEs periodically. The reinitialization is similar
to a required step in the closest point method. However, due to the
special form of the initial conditions and the equations, the stabilization
step is required only infrequently. 

We remark that the proposed formulation is not limited to implementation
using Cartesian grids. It is also convenient for design of finite
element methods on non-body fitted mesh. Furthermore, the proposed
formulation, which considers minimization of the variational principles
instead of tackling directly the Euler-Lagrange equations, allows
for the possibility of applying numerical optimization algorithms
to the discretized variational principles. A potential advantage of
such strategy include better preservation of invariances of the systems,
see for example the variational integrator \cite{Marsden-West:Acta-Numerica},
and the flexibility in dealing with nonlinear degenerate systems such
as the total variations of a function on surfaces \cite{chambolle-TV-direct-discretization}.
This direction will be pursued in a future work.

\section*{Acknowledgements}

Tsai thanks the National Center for Theoretical Sciences, Taiwan for
support of his visits, during which this work was initiated and completed.
Tsai was partially supported by NSF grants DMS-1318975 and DMS-1620473.
Chu was partially supported by MOST grants 105-2115-M-007 -004 and
106-2115-M-007 -002.

\appendix

\section{Appendix}

\subsection{Extension of surface gradient and surface divergence in $\mathcal{\mathbb{R}}^{3}$\label{subsec:Extension-of-surface}}

Let $\Omega\subset\mathbb{R}^{3}$ be a bounded open set with $C^{2}$
boundary $\Gamma=\partial\Omega$. For simplicity, we assume there
exists a signed-distance function $d:\mathbb{R}^{3}\rightarrow\mathbb{R}$
such that $\Gamma$ is the zero-level set of $d$. Without loss of
generality, we also assume that $d<0$ on the interior of $\Gamma$
and $d>0$ on the exterior. The normal vector ${\bf n}(x)=\nabla d(x)$
is the unit outer normal vector field of $\Gamma$ and the projection
$\Pi=I-{\bf n}\otimes{\bf n}$ which maps vectors in $\mathbb{R}^{3}$
onto the tangent space of $\Gamma$ at $x$. For any smooth function
$u$ defined on $\Gamma$, the tangent gradient of $u$ is defined
by

\[
\nabla_{\Gamma}u(x)=\Pi\nabla\tilde{{u}}(x),\quad\forall x\in\Gamma,
\]
where $\tilde{{u}}$ is any $C^{1}$ extension of $u(x)$ in a neighborhood
of $x$. For any vector field $F$ defined on $\Gamma$, the surface
divergence $\nabla_{\Gamma}\cdot F$ is defined analogously. Recall
that $\eta$-level set of $d$ is $\Gamma_{\eta}=\{x\in\mathbb{R}^{3}\,|\,d(x)=\eta\}$
and $T_{\epsilon}=\cup_{|\eta|<\epsilon}\Gamma_{\eta}$ is the narrowband
of $\Gamma$. The following theorem relates normal extension of the
surface gradient and divergence to the Eulerian gradient and divergence
of normal-extended function.
\begin{thm}
\label{thm:A1}Suppose $u$ is a smooth function defined on $\Gamma$.
Let $P_{\Gamma}$ denote the closest point mapping and \emph{$\overline{u}$
denote the constant-along-normal extension }of u in $T_{\epsilon}$.
Then for any $z\in T_{\epsilon}$ , we have 
\begin{equation}
\overline{\nabla_{\Gamma}u(z)}\equiv\nabla_{\Gamma}u\circ P_{\Gamma}(z)=A(z;\mu)\nabla\overline{u}(z),\label{eq:grad relation}
\end{equation}
where $A(z;\mu)=A_{0}(z)+\mu A_{1}(z),$
\begin{align*}
A_{0}(z):= & \sigma_{1}^{-1}\mathbf{t}_{1}\otimes\mathbf{t}_{1}+\sigma_{2}^{-1}\mathbf{t}_{2}\otimes\mathbf{t}_{2},\\
A_{1}(z):= & \mathbf{n}\otimes\mathbf{n},
\end{align*}
where $\mathbf{t}_{1}$, $\mathbf{t}_{2}$ are the two orthonormal
tangent vectors corresponding to the directions that yield the principle
curvatures of $\Gamma$, $\mathbf{n}$ is the unit normal vector of
$\Gamma$, $\sigma_{1}$ and $\sigma_{2}$ are two largest singular
values of $DP_{\Gamma}$ and $\mu$ is any real number.

Suppose $F$ is a smooth vector field defined on $\Gamma$. Let\emph{
$\overline{F}$ denote its constant-along-normal extension} in $T_{\epsilon}$.
Then for any $z\in T_{\epsilon}$ , we have 
\begin{equation}
\overline{\nabla_{\Gamma}\cdot F(z)}\equiv(\nabla_{\Gamma}\cdot F)\circ P_{\Gamma}(z)=\sigma_{1}^{-1}\sigma_{2}^{-1}\nabla\cdot(B(z;\mu)\overline{F}(z))=J^{-1}\nabla\cdot(B(z;\mu)\overline{F}(z)),\label{eq:div relation}
\end{equation}
where $B(z;\nu)=B_{0}(z)+\nu B_{1}(z),$
\begin{align*}
B_{0}(z):= & \sigma_{2}\mathbf{t}_{1}\otimes\mathbf{t}_{1}+\sigma_{1}\mathbf{t}_{2}\otimes\mathbf{t}_{2}=\sigma_{1}\sigma_{2}A_{0}(z)=JA_{0}(z),\\
B_{1}(z):= & \mathbf{n}\otimes\mathbf{n},
\end{align*}
where $\nu$ is any real number. In particular, if we choose $\nu=J\mu$,
then $B(z;\nu)=JA(z;\mu)$.
\end{thm}
\begin{proof}
Let $(s_{1},s_{2})$ be a local coordinate system for $\Gamma$ such
that ${\bf t}_{1}=\frac{{\partial\Gamma(s_{1,}s_{2})}}{\partial s_{1}}$
and ${\bf t}_{2}=\frac{{\partial\Gamma(s_{1,}s_{2})}}{\partial s_{2}}$
are orthogonal unit eigenvectors of $D^{2}d$ corresponding to two
principle curvatures $\kappa_{1}$ and $\kappa_{2}$ respectively.
Notice that $(s_{1},s_{2},\eta)$ forms curvilinear coordinates on
$T_{\epsilon}$ with the coordinate transformation

\begin{equation}
z(s_{1},s_{2},\eta)=\Gamma(s_{1},s_{2})+\eta{\bf n}(s_{1},s_{2}),\label{eq:surface_eta}
\end{equation}
where ${\bf n}(s_{1,}s_{2})={\bf n}(\Gamma(s_{1,}s_{2}))$. By using
$\frac{\partial{\bf n}}{\partial s_{i}}=\kappa_{i}{\bf t}_{i}$, we
obtain $\frac{\partial z}{\partial s_{1}}=(1+\eta\kappa_{1}(s_{1,}s_{2})){\bf t}_{1}$
, $\frac{\partial z}{\partial s_{2}}=(1+\eta\kappa_{2}(s_{1,}s_{2})){\bf t_{2}}$
and $\frac{\partial z}{\partial\eta}={\bf n}$. By formula of gradient
in orthogonal curvilinear coordinate systems, we have

\[
\nabla v=(1+\eta\kappa_{1})^{-1}\frac{\partial v}{\partial s_{1}}{\bf t_{1}}+(1+\eta\kappa_{2})^{-1}\frac{\partial v}{\partial s_{1}}{\bf t_{2}}+\frac{\partial v}{\partial\eta}{\bf n}=(1+\eta\kappa_{1})^{-1}\frac{\partial v}{\partial s_{1}}{\bf t_{1}}+(1+\eta\kappa_{2})^{-1}\frac{\partial v}{\partial s_{1}}{\bf t_{2}}
\]
It follows the tangential gradient of $u$ is
\[
\nabla_{\Gamma}u(P_{\Gamma}z)=\frac{\partial v}{\partial s_{1}}{\bf t_{1}}+\frac{\partial v}{\partial s_{1}}{\bf t_{2}}=((1+d(z)\kappa_{1}(P_{\Gamma}z)){\bf t}_{1}\otimes{\bf t}_{1}+(1+d(z)\kappa_{2}(P_{\Gamma}z)){\bf t}_{2}\otimes{\bf t}_{2})\nabla v(z)
\]
By using the fact $(1+d(z)\kappa_{i}(P_{\Gamma}z))=(1-d(z)\kappa_{i}(z))^{-1}=\sigma_{i}$
and ${\bf n}\cdot\nabla v=0$, we prove \eqref{eq:grad relation}. 

First notice that ${\bf n}\cdot\overline{F}=0$ since $\overline{F}$
is the normal extension of $F.$ Let $F^{1}$ and $F^{2}$ denote
the components of $F$ in ${\bf t_{1},}{\bf t_{2}}$ respectively.
That is, $F(x)=F^{1}(x){\bf t_{1}}+F^{2}(x){\bf t_{2}}$ for $x\in\Gamma$.
By formula of divergence in orthogonal curvilinear coordinate systems,
we have
\[
\nabla\cdot\overline{F}=(1+\eta\kappa_{1})^{-1}(1+\eta\kappa_{2})^{-1}(\frac{\partial((1+\eta\kappa_{2})\overline{F^{1}})}{\partial s_{1}}+\frac{\partial((1+\eta\kappa_{1})\overline{F^{2}})}{\partial s_{2}}+\frac{\partial((1+\eta\kappa_{1})(1+\eta\kappa_{2})G_{{\bf n}})}{\partial\eta})
\]
Therefore it follows

\[
\nabla_{\Gamma}\cdot F=\frac{\partial F_{{\bf t_{1}}}}{\partial s_{1}}+\frac{\partial F_{{\bf t}_{2}}}{\partial s_{2}}=(1+\eta\kappa_{1})(1+\eta\kappa_{2})\nabla\cdot(B_{0}G).
\]
By ${\bf n}\cdot\overline{F}=0$ , we have $B_{1}(z;\nu)\overline{F}=0$.
This shows \eqref{eq:div relation} holds. 
\end{proof}
\begin{rem*}
In fact, by using ${\bf n}\cdot\overline{F}=0$ and $\frac{\partial\overline{F}_{{\bf t_{1}}}}{\partial\eta}\text{=}\frac{\partial\overline{F}_{{\bf t_{2}}}}{\partial\eta}=0$,
we can choose more general $B_{1}$ as follow 
\[
B_{1}(z;\nu):=\mathbf{\nu_{1}n}\otimes\mathbf{n}+\nu_{2}\mathbf{t_{1}}\otimes\mathbf{n}+\mathbf{\nu_{3}t_{2}}\otimes\mathbf{n}+\nu_{4}{\bf n}\otimes{\bf t_{1}}+\nu_{5}{\bf n}\otimes{\bf t_{2}},
\]
where $\nu_{i}$ are arbitrary real numbers. 
\end{rem*}
\begin{rem*}
If $\Gamma$ is a curve in $\mathbb{R}^{2},$ then it can be shown
analogously that 

\begin{equation}
\nabla_{\Gamma}u(P_{\Gamma}z)=(\sigma^{-1}{\bf t}\otimes{\bf t}+\mu{\bf n}\otimes{\bf n})\nabla v(z),
\end{equation}
and 
\end{rem*}
\begin{equation}
\nabla_{\Gamma}\cdot F=\sigma^{-1}\nabla\cdot({\bf {\bf t}\otimes{\bf t}+\nu n}\otimes{\bf n)}\,G).
\end{equation}

\subsection{Second order finite difference approximation for normal extension
of surface Laplacian}

In this paper, all numerical experiments are done by second order
finite difference schemes. We present the detail about finite difference
scheme to approximate the normal extension of surface Laplacian $\Delta_{\Gamma}u$.
For simplicity, we demonstrate 2 dimensional case and assume that
$\Delta x=\Delta y=h$, but it can be easily generalized to higher
dimensional cases with nonuniform Cartesian grids. Recall that the
normal extended equation for surface Laplacian is given by

\[
\sigma^{-1}\nabla\cdot A\nabla v=\sigma^{-1}[(A^{11}v_{x})_{x}+(A^{12}v_{y})_{x}+(A^{21}v_{x})_{y}+(A^{22}v_{y})_{y}],
\]
where $A=\sigma^{-1}{\bf t}\otimes{\bf t}+\mu{\bf n\otimes}{\bf n}=\left[\begin{array}{cc}
A^{11} & A^{12}\\
A^{21} & A^{22}
\end{array}\right]$. We use central difference to approximate all terms as following
\begin{align*}
[(A^{11}v_{x})_{x}]_{ij} & \text{=}\frac{1}{h}([A^{11}v_{x}]_{i+\frac{1}{2},j}-[A^{11}v_{x}]_{i-\frac{1}{2},j})+O(h^{2})\\
 & =\frac{1}{h^{2}}(A_{i+\frac{1}{2},j}^{11}(v_{i+1,j}-v_{i,j})-A_{i-\frac{1}{2},j}^{11}(v_{i,j}-v_{i-1,j}))+O(h^{2}),
\end{align*}
\begin{align*}
[(A^{12}v_{y})_{x}]_{ij} & =\frac{1}{2h}([A^{12}v_{y}]_{i+1,j}-[A^{12}v_{y}]_{i-1,j})+O(h^{2})\\
 & =\frac{1}{\text{4}h^{2}}(A_{i+1,j}^{12}(v_{i+1,j+1}-v_{i+1,j-1})-A_{i-1,j}^{12}(v_{i-1,j+1}-v_{i-1,j-1}))+O(h^{2}),
\end{align*}
\begin{align*}
[(A^{21}v_{x})_{y}]_{ij} & =\frac{1}{2h}([A^{21}v_{x}]_{i,j+1}-[A^{21}v_{x}]_{i,j-1})+O(h^{2})\\
 & =\frac{1}{\text{4}h^{2}}(A_{i,j+1}^{21}(v_{i+1,j+1}-v_{i-1,j+1})-A_{i,j-1}^{12}(v_{i+1,j-1}-v_{i-1,j-1}))+O(h^{2}),
\end{align*}
\begin{align*}
[(A^{22}v_{y})_{y}]_{ij} & \text{=\ensuremath{\frac{1}{h}}}([A^{22}v_{y}]_{i,j+\frac{1}{2}}-[A^{22}v_{y}]_{i,j-\frac{1}{2}})+O(h^{2})\\
 & =\frac{1}{h^{2}}(A_{i,j+\frac{1}{2}}^{22}(v_{i,j+1}-v_{i,j})-A_{i,j-\frac{1}{2}}^{22}(v_{i,j}-v_{i,j-1}))+O(h^{2}).
\end{align*}
If $v_{l,k}$ is taking value at a ghost node, replace $v_{l,k}$
by linear combination of other $v_{i,j}$ at inner nodes as discussed
in Section \ref{subsec:proposed-BC}.

\bibliographystyle{abbrv}
\bibliography{refs}

\begin{thebibliography}{10}

\bibitem{ahmed2011third}
S.~Ahmed, S.~Bak, J.~McLaughlin, and D.~Renzi.
\newblock A third order accurate fast marching method for the eikonal equation
  in two dimensions.
\newblock {\em SIAM Journal on Scientific Computing}, 33(5):2402--2420, 2011.

\bibitem{BARRETT:1984:imanum/4.3.309}
J.~W. Barrett and C.~M. Elliott.
\newblock A finite-element method for solving elliptic equations with Neumann
  data on a curved boundary using unfitted meshes.
\newblock {\em {IMA Journal of Numerical Analysis}}, 4:309--325, 1984.

\bibitem{bertalmio2001variational}
M.~Bertalm{\i}o, L.-T. Cheng, S.~Osher, and G.~Sapiro.
\newblock Variational problems and partial differential equations on implicit
  surfaces.
\newblock {\em Journal of Computational Physics}, 174(2):759--780, 2001.

\bibitem{chambolle-TV-direct-discretization}
A.~Chambolle.
\newblock An algorithm for total variation minimization and applications.
\newblock {\em Journal of Mathematical imaging and vision}, 20(1-2):89--97,
  2004.

\bibitem{Redistance_ChengTsai}
L.~Cheng and R.~Tsai.
\newblock Redistancing by flow of time dependent eikonal equation.
\newblock {\em Journal of Computational Physics}, 227, 2008.

\bibitem{MR3249369}
K.~Deckelnick, C.~M. Elliott, and T.~Ranner.
\newblock Unfitted finite element methods using bulk meshes for surface partial
  differential equations.
\newblock {\em SIAM J. Numer. Anal.}, 52(4):2137--2162, 2014.

\bibitem{MR2485433}
A.~Demlow.
\newblock Higher-order finite element methods and pointwise error estimates for
  elliptic problems on surfaces.
\newblock {\em SIAM J. Numer. Anal.}, 47(2):805--827, 2009.

\bibitem{MR2285862}
A.~Demlow and G.~Dziuk.
\newblock An adaptive finite element method for the {L}aplace-{B}eltrami
  operator on implicitly defined surfaces.
\newblock {\em SIAM J. Numer. Anal.}, 45(1):421--442, 2007.

\bibitem{MR976234}
G.~Dziuk.
\newblock Finite elements for the {B}eltrami operator on arbitrary surfaces.
\newblock In {\em Partial differential equations and calculus of variations},
  volume 1357 of {\em Lecture Notes in Math.}, pages 142--155. Springer,
  Berlin, 1988.

\bibitem{dziuk2013finite}
G.~Dziuk and C.~M. Elliott.
\newblock Finite element methods for surface PDEs.
\newblock {\em Acta Numerica}, 22:289--396, 2013.

\bibitem{Evans-Gariepy}
L.~C. Evans and R.~F. Gariepy.
\newblock {\em Measure theory and fine properties of functions}.
\newblock Studies in Advanced Mathematics. CRC Press, Boca Raton, FL, 1992.

\bibitem{greer2006improvement}
J.~B. Greer.
\newblock An improvement of a recent Eulerian method for solving PDEs on
  general geometries.
\newblock {\em Journal of Scientific Computing}, 29(3):321--352, 2006.

\bibitem{gustafsson-kreiss-oliger}
B.~Gustafsson, H.-O. Kreiss, and J.~Oliger.
\newblock {\em Time dependent problems and difference methods}, volume~24.
\newblock John Wiley \& Sons, 1995.

\bibitem{imbert2017pseudo}
L.-M. Imbert-G{\'e}rard and L.~Greengard.
\newblock Pseudo-spectral methods for the Laplace-Beltrami equation and the
  Hodge decomposition on surfaces of genus one.
\newblock {\em Numerical Methods for Partial Differential Equations},
  33(3):941--955, 2017.

\bibitem{jiang2000weighted}
G.-S. Jiang and D.~Peng.
\newblock {Weighted ENO} schemes for {Hamilton}--{Jacobi} equations.
\newblock {\em SIAM Journal on Scientific computing}, 21(6):2126--2143, 2000.

\bibitem{kreiss-petersson-ystrom2004}
H.-O. Kreiss, N.~A. Petersson, and J.~Ystr{\"o}m.
\newblock Difference approximations of the Neumann problem for the second order
  wave equation.
\newblock {\em SIAM Journal on Numerical Analysis}, 42(3):1292--1323, 2004.

\bibitem{KTT}
C.~Kublik, N.~Tanushev, and R.~Tsai.
\newblock An implicit interface boundary integral method for {Poisson}'s
  equation on arbitrary domains.
\newblock {\em Journal of Computational Physics}, 247, 2013.

\bibitem{kublik2016integration}
C.~Kublik and R.~Tsai.
\newblock Integration over curves and surfaces defined by the closest point
  mapping.
\newblock {\em Research in the mathematical sciences}, 3(3), 2016.

\bibitem{macdonald_ruuth09}
C.~B. Macdonald and S.~J. Ruuth.
\newblock The implicit {C}losest {P}oint {M}ethod for the numerical solution of
  partial differential equations on surfaces.
\newblock {\em SIAM J. Sci. Comput.}, 31(6):4330--4350, 2009.

\bibitem{Marsden-West:Acta-Numerica}
J.~E. Marsden and M.~West.
\newblock Discrete mechanics and variational integrators.
\newblock {\em Acta Numerica}, pages 357--514, 2001.

\bibitem{olshanskii2016narrow}
M.~Olshanskii and D.~Safin.
\newblock A narrow-band unfitted finite element method for elliptic PDEs posed
  on surfaces.
\newblock {\em Mathematics of Computation}, 85(300):1549--1570, 2016.

\bibitem{olshanskii2010finite}
M.~A. Olshanskii and A.~Reusken.
\newblock A finite element method for surface PDEs: matrix properties.
\newblock {\em Numerische Mathematik}, 114(3):491--520, 2010.

\bibitem{MR2551197}
M.~A. Olshanskii, A.~Reusken, and J.~Grande.
\newblock A finite element method for elliptic equations on surfaces.
\newblock {\em SIAM J. Numer. Anal.}, 47(5):3339--3358, 2009.

\bibitem{o2017second}
M.~O'Neil.
\newblock Second-kind integral equations for the Laplace-Beltrami problem on
  surfaces in three dimensions.
\newblock {\em arXiv preprint arXiv:1705.00069}, 2017.

\bibitem{LevelSet_OsherFedkiw}
S.~Osher and R.~Fedkiw.
\newblock {\em Level set methods and dynamic implicit surfaces}.
\newblock Springer, 2000.

\bibitem{osher_sethian88}
S.~Osher and J.~A. Sethian.
\newblock Fronts propagating with curvature dependent speed: Algorithms based
  on Hamilton-Jacobi formulations.
\newblock {\em J. Comp. Phys.}, 79:12--49, 1988.

\bibitem{reuter2006laplace}
M.~Reuter, F.-E. Wolter, and N.~Peinecke.
\newblock Laplace--Beltrami spectra as 'Shape-DNA' of surfaces and solids.
\newblock {\em Computer-Aided Design}, 38(4):342--366, 2006.

\bibitem{rustamov2007laplace}
R.~M. Rustamov.
\newblock Laplace-Beltrami eigenfunctions for deformation invariant shape
  representation.
\newblock In {\em Proceedings of the fifth Eurographics symposium on Geometry
  processing}, pages 225--233. Eurographics Association, 2007.

\bibitem{ruuth_merriman08}
S.~J. Ruuth and B.~Merriman.
\newblock A simple embedding method for solving partial differential equations
  on surfaces.
\newblock {\em J. Comput. Phys.}, 227(3):1943--1961, 2008.

\bibitem{sethian1999fast}
J.~A. Sethian.
\newblock Fast marching methods.
\newblock {\em SIAM review}, 41(2):199--235, 1999.

\bibitem{Tsai:JCP2002}
Y.-h.~R. Tsai.
\newblock Rapid and accurate computation of the distance function using grids.
\newblock {\em J. Comput. Phys.}, 178(1):175--195, 2002.

\bibitem{vogl2016curvature}
C.~J. Vogl.
\newblock The curvature-augmented closest point method with vesicle
  inextensibility application.
\newblock {\em arXiv preprint arXiv:1610.03932}, 2016.

\bibitem{xu2006level}
J.-J. Xu, Z.~Li, J.~Lowengrub, and H.~Zhao.
\newblock A level-set method for interfacial flows with surfactant.
\newblock {\em Journal of Computational Physics}, 212(2):590--616, 2006.

\bibitem{xu2012level}
J.-J. Xu, Y.~Yang, and J.~Lowengrub.
\newblock A level-set continuum method for two-phase flows with insoluble
  surfactant.
\newblock {\em Journal of Computational Physics}, 231(17):5897--5909, 2012.

\bibitem{xu2003eulerian}
J.-J. Xu and H.-K. Zhao.
\newblock An Eulerian formulation for solving partial differential equations
  along a moving interface.
\newblock {\em Journal of Scientific Computing}, 19(1):573--594, 2003.

\bibitem{zhang2006high}
Y.-T. Zhang, H.-K. Zhao, and J.~Qian.
\newblock High order fast sweeping methods for static {Hamilton}--{Jacobi}
  equations.
\newblock {\em Journal of Scientific Computing}, 29(1):25--56, 2006.

\end{thebibliography}

\end{document}